\documentclass[usletter]{amsart} 

\usepackage{amsmath,amsthm,amssymb,amsfonts,mathrsfs,color,hyperref, mathtools,crop, graphicx, enumitem, todonotes, thm-restate}
\usepackage[left=3cm, right=3cm, bottom=3cm]{geometry}

\theoremstyle{plain}
\begingroup
\newtheorem{theorem}{Theorem}[section]
\newtheorem*{theorem*}{Theorem}
\newtheorem*{"theorem"}{``Theorem''}
\newtheorem{corollary}[theorem]{Corollary}
\newtheorem{proposition}[theorem]{Proposition}

\newtheorem{lemma}[theorem]{Lemma}
\endgroup

\theoremstyle{definition}
\begingroup

\endgroup

\theoremstyle{remark}
\begingroup

\newtheorem{example}[theorem]{Example}
\endgroup 

\numberwithin{equation}{section}
\newcommand{\N}{\mathbb N}

\newcommand{\R}{\mathbb R} 
\newcommand{\E}{{\mathbb E}}

\newcommand{\F}{{\mathcal F}}

\newcommand{\LRa} {\Leftrightarrow}
\newcommand{\Ra} {\Rightarrow}

\newcommand{\dx}{\,\mathrm{d}x}

\newcommand{\ds}{\,\mathrm{d}s}
\newcommand{\dt}{\,\mathrm{d}t}

\newcommand{\eps}{\varepsilon}
\newcommand{\average}{{\mathchoice {\kern1ex\vcenter{\hrule height.4pt
width 6pt depth0pt} \kern-9.7pt} {\kern1ex\vcenter{\hrule
height.4pt width 4.3pt depth0pt} \kern-7pt} {} {} }}

\allowdisplaybreaks

 \makeatletter
\@namedef{subjclassname@2020}{2020 Mathematics Subject Classification}
\makeatother

\newcommand{\sw}[1]{{\color{black}#1}}
 
\begin{document}

\title[On differences between gradient flows in finite and infinite dimension]{
A qualitative difference between gradient flows of convex functions in finite- and infinite-dimensional Hilbert spaces
%Sharp energy decay asymptotics in convex optimization
}

\author{Jonathan W. Siegel}
\address{Jonathan W. Siegel\\
Department of Mathematics\\
Texas A\&M University\\
155 Ireland Street\\
College Station, TX 77840
}
\email{jwsiegel@tamu.edu}

\author{Stephan Wojtowytsch}
\address{Stephan Wojtowytsch\\
University of Pittsburgh\\
Department of Mathematics\\
Thackeray Hall\\
Pittsburgh, PA 15221}
\email{s.woj@pitt.edu}

\date{\today}


\begin{abstract}
We consider gradient flow/gradient descent and heavy ball/accelerated gradient descent optimization for convex objective functions. In the gradient flow case, we prove the following:
\begin{enumerate}
    \item If $f$ does not have a minimizer, the convergence $f(x_t) \to \inf f$ can be arbitrarily slow.
    \item If $f$ does have a minimizer, the excess energy $f(x_t) - \inf f$ is integrable/summable in time. In particular, $f(x_t) - \inf f = o(1/t)$ as $t\to\infty$. 
    \item In Hilbert spaces, this is optimal: $f(x_t) - \inf f$ can decay to $0$ as slowly as any given function which is monotone decreasing and integrable at $\infty$, even for a fixed quadratic objective.
    \item In finite dimension (or more generally, for all gradient flow curves of finite length), this is not optimal: We prove that there are convex monotone decreasing integrable functions $g(t)$ which decrease to zero slower than $f(x_t)-\inf f$ for the gradient flow of any convex function on $\R^d$. For instance, we show that every gradient flow $x_t$ of a convex function $f$ in finite dimension satisfies $\liminf_{t\to\infty} \big(t\cdot \log^2(t)\cdot \big\{f(x_t) -\inf f\big\}\big)=0$.
\end{enumerate}
This improves on the commonly reported $O(1/t)$ rate and provides a sharp characterization of the energy decay law. We also note that it is impossible to establish a rate $O(1/(t\phi(t))$ for any function $\phi$ which satisfies $\lim_{t\to\infty}\phi(t) = \infty$, even asymptotically.

Similar results are obtained in related settings for (1) discrete time gradient descent, (2) stochastic gradient descent with multiplicative noise and (3) the heavy ball ODE. In the case of stochastic gradient descent, the summability of $\mathbb E[f(x_n) - \inf f]$ is used to prove that $f(x_n)\to \inf f$ almost surely -- an improvement on the convergence almost surely up to a subsequence which follows from the $O(1/n)$ decay estimate.
\end{abstract}

\maketitle

% \setcounter{tocdepth}{1}
% \tableofcontents

\section{Introduction}

In this note, we discuss two approaches in gradient-based optimization for convex objective functions: Steepest descent and the momentum method. They are described by the gradient flow ODE $\dot x = - \nabla f(x)$ and the heavy ball ODE $\ddot x = - \alpha\,\dot x - \nabla f(x)$ respectively. The heavy ball ODE is Newton's second for a particle of mass $m=1$ under the influence of a potential force $-\nabla f$ and a Stokes type friction with coefficient $\alpha$. 

If the objective function $f$ is merely convex, but not strongly convex, it is often favorable to let the coefficient of friction decay to zero, as the objective function can be very flat at its minimum. Constant friction in this setting dissipates kinetic energy too quickly, resulting in very slowly converging trajectories.  \cite{nesterov1983method, su2014differential} illustrate that the scaling $\alpha(t):= \frac{\alpha^*}t$ with $\alpha^*\geq 3$ balances the desirable effects of friction (extracting sufficient energy to dampen around the minimizer) against the undesirable (slowing down dynamics on the path to the minimizer).

In real applications, it is often not necessary to find a minimizer of $f$, but a point of low objective value. We therefore focus on studying the risk decay curves $g(t) = f(x(t)) - \inf f$ for gradient flows and heavy balls.

The rates which are commonly reported for convex optimization are $O(1/t)$ for gradient flows and $O(1/t^2)$ for the heavy ball ODE. In discrete time convex optimization, the rate $O(1/k^2)$ is optimal for any iterative method for which $x_{n+1} -x_n \in \mathrm{span}\{\nabla f(x_0), \dots, \nabla f(x_n)\}$ \cite[Section 2.1.2]{nesterov2003introductory}, and it is achieved by Nesterov's accelerated gradient method. However, at any finite iteration $k$, a different (convex, quadratic) function on a space of dimension $\geq 2k$ is used to construct a lower bound. For a fixed convex function, Nesterov's method decreases the objective function as $o(1/k^2)$ \cite{attouch2016rate}, but generally not as $O(1/k^{2+\eps})$ for any $\eps>0$.  \cite{su2014differential} demonstrate that
\begin{equation}\label{eq integrability nesterov intro}
\int_0^\infty t\big(f(x(t)) - \inf f\big) \dt < +\infty.
\end{equation}
In this work, we show that
\begin{enumerate}
    \item The characterization of the risk decay by \eqref{eq integrability nesterov intro} is essentially sharp even for a {\em fixed} quadratic function on a separable Hilbert space.
    
    \item For gradient descent, the sharp characterization is given by the corresponding integrability condition that
    \[
    \int_0^\infty \big(f(x(t)) - \inf f\big) \dt < + \infty.
    \]
    
    \item In {\em finite-dimensional Hilbert} spaces, gradient flows decrease the excess objective value $f(x(t)) - \inf f$ faster in the sense that a stricter decay condition holds than mere integrability. In particular, we show that
    \[
    \liminf_{t\to\infty} t\cdot \log^2(t) \cdot \big(f(x(t)) - \inf f\big) = 0.
    \]
    However, we demonstrate that statement cannot be strengthend by replacing the lower limit with an upper limit. The upper limit cannot be improved beyond the statement that $f(x(t)) - \inf f = o(1/t)$, even in dimension $d=1$.

    The improvement is based on the fact that gradient flow trajectories of convex objective functions have finite length in finite-dimensional Hilbert spaces. In the infinite-dimensional case, this may not be true.

    \item We extend the integrability condition to gradient descent in discrete time both in the deterministic setting and the stochastic setting with noise which scales in a multiplicative fashion. Here we show that the SGD iterates satisfy
    \[
    \sum_{n=0}^\infty \E\big[ f(x_n) - \inf f\big] < +\infty 
    \]
    and deduce that $f(x_n)\to \inf f$ almost surely. Absent summability, the statement would remain true almost surely only along a subsequence.

    \item All previous results are valid under the assumption that $f$ is not convex {\em and} that there exists a point $x^*$ such that $f(x^*) = \inf f$. We show that without this assumption, the decay $f(x(t)) \to \inf f$ may be arbitrarily slow for both gradient flow and heavy ball ODE.
\end{enumerate}

The improvement to the bounds is qualitative and asymptotic in nature. While quantitative and non-asymptotic bounds have great benefits, the constants involved in the bounds are rarely available in practice. We therefore maintain that a sharp qualitative understanding of the convergence towards a minimum value is helpful.

Convex functions without minimizers are very common for example in applications in machine learning for classification if the cross-entropy loss function is used. 

We believe that some of these results may be familiar to experts in the field, but we have been unable to find references for many of them. A principal goal of this work is to address this gap and provide a simple introduction to sharp results on gradient flows and accelerated gradient methods.

The article is structured as follows: Continuous time gradient flows are discussed in Section \ref{section gradient flow} with special attention to the impact of finite dimension and the existence of minimizers. Corresponding results are obtained for gradient descent and stochastic gradient descent in Section \ref{section gradient descent}. Momentum methods are only considered in continuous time in Section \ref{section heavy ball}, but references to corresponding discrete time results are provided in the appropriate places. A technical results concerning convex functions whose derivative is $L^{1/2}$-integrable on $(0,\infty)$ is postponed until the appendix.

\subsection{Significance in Machine Learning}

Our main motivation for the study of gradient-based optimizers is the recent popularity of simple first order optimization algorithms in machine learning and specifically in deep learning. They have been used with great success to minimize high-dimensional functions like
\[
L(w) = \E_{(x,y)\sim \mu} \big[ \|h(w,x) - y\|^2\big]
\quad\text{or the discretization }\quad 
L_n(w) = \frac1n \sum_{i=1}^n \big\|h(w,x_i) - y_i\big\|^2.
\]
Here $h$ denotes a parametrized function (e.g.\ a neural network) with parameters (`weights') $w\in\R^m$ and data $x\in\R^d$. The expression inside the expectation or sum can be more general -- in classification, the cross-entropy loss
\[
\ell_{ce}\big(h(w,x), \,y\big) := -\log\left(\frac{\exp(h(w,x) \cdot y)}{\sum_{j=1}^k \exp(h_j(w,x))}\right)
\]
is more popular than the `mean squared error' (MSE)-loss = $\ell^2$-loss \cite{hui2020evaluation}. Here $y = e_l$ is a vector corresponding to the label $l\in\{1,\dots,k\}$ of the point $x$. The `loss function' inside the expectation/empirical average $L(h) = \E[\ell(h(x),y)]$ is typically convex, but not necessarily strictly convex. For instance, cross-entropy loss fails to be strictly convex since $\ell(h+\lambda(1,\dots,1), \,y) \equiv \ell(h,y)$ for all $\lambda\in\R$, i.e.\ the loss function is constant in one direction. Additionally, $\inf_h\ell_{ce}(h, y) = \lim_{\lambda\to\infty} \ell_{ce}(\lambda y, y) =0$, but $\ell_{ce}(h,y)>0$ for any $h, y$, so $\ell_{ce}$ does not admit minimizers.

If, for instance, $h(W, x) = Wx$ is a linear model with $W\in \R^{k\times d}$, then $L_n$ is convex. If additionally there exists $W^*$ such that all points are classified correctly in the sense that
\[
h(W^*,x_i) \cdot y_i > \max_{e_j\neq y_i} h(W^*,x_i)\cdot e_j,
\]
then there exists no minimizer of $L_n$ since $L_n(W)>0$ for all $W$ but $\lim_{\lambda\to\infty}L_n(\lambda W^*) = 0$. The second point remains valid if the parametrized function class is not linear, but merely a cone (e.g.\ a class of neural networks).

We remark that cross-entropy has other favorable properties which distinguish it from the worst case scenarios discussed above. For instance, $\ell_{ce}$ has exponential tails, i.e.\ it vanishes very quickly at infinity. Such properties can be captured for example in the language of Polyak-\L ojasiewicz (PL) conditions, which have been used very successfully to study gradient flows (but not heavy ball methods). Notably, mere convexity is not enough to study gradient-based optimization for objective functions without minimizers.

For parameterized functions $h(w,\cdot)$ which depend on their parameters in a non-linear fashion (such as neural networks) the functional $L_n(w):= \frac1n\sum_{i=1}^n \ell \big(h(w,x_i), y_i\big)$ generally fails to be convex, even if the loss function $\ell$ is convex in the first argument:
\begin{itemize}
    \item In the underparametrized regime of neural network learning, Safran and Shamir proved that the loss landscape generally contains many non-optimal local minimizers \cite{DBLP:conf/icml/SafranS18}.
    \item In the overparametrized regime, the set of minimizing parameters generally is a high-dimensional submanifold of the parameter space. Negative Hessian eigenvalues have been observed numerically close to the set of minimizers \cite{sagun2016eigenvalues, sagun2017empirical, alain2019negative}, and their existence has been justified theoretically in \cite{wojtowytsch2021stochasticdiscrete}.
\end{itemize}

Despite this non-convexity of the loss landscape, it has been observed that the weights of a neural network may remain in a `good region' and closely follow the trajectory of optimizing a linear model which is obtained by linearizing the neural network at the law of its initialization. This `neural tangent kernel' (NTK) was considered for gradient descent in \cite{weinan2019comparative, jacot2018neural, du2018gradient, du2019gradient, arora2019exact} and for momentum-based optimization in \cite{liu2022provable}. The NTK is linear in its parameters, and thus the loss landscape associated to the parameter optimziation process is convex if $\ell$ is a convex function. This suggests that locally around a good initialization, the optimization landscape looks similar to that of a convex function. The crucial observation is that trajectories of gradient flows remain in this `good' region for all time under suitable conditions.

Even globally, for very wide networks there are no strict local minimizers which are not global minimizers \cite{venturi2018spurious, wang2021hidden}. In this way, convex optimization informs the intuition of parameter optimization in deep learning, at least in a heavily overparametrized regime. 

Finding exact minimizers of the empirical loss function $L_n$ is not always attractive in deep learning, where the true goal is to find minimizers of the (unknown) function $L$. The question whether a parameter $w$ `generalizes' well (performs well on previously unseen data sampled from the same distribution which generated the training data $(x_i, y_i)$, $i=1,\dots,n$) is generally considered more important than how close it is to the true optimal parameter $w^*$. At the optimal parameter for a given data sample, we may `overfit' to random noise in the training data with possibly disastrous implications for generalization.

This motivates us to primarily consider the convergence of $f(x_t)$ rather than that of $x_t$.

\subsection{Technical tools}

Most proofs in the following are elementary. We recall a statement which will be used frequently throughout the article.

\begin{lemma}\label{lemma monotone integrable decay}
    Let $f:(0,\infty) \to (0,\infty)$ be a monotone decreasing function such that $\int_0^\infty f(x)\dx <\infty$. Then
    \[
    f(x) \leq \frac{2\int_0^\infty f(t)\dt}x\qquad\text{and }
    \quad
    \lim_{x\to\infty} x\cdot f(x) =0.
    \]
\end{lemma}

\begin{proof}
    Since $f$ is decreasing, we find that 
    \[
    \int_{x/2}^\infty f(t)\dt \geq \int_{x/2}^x f(t)\dt \geq \int_{x/2}^x f(x)\dt = \frac{x}2 \cdot f(x)\geq 0
    \]
    On the other hand, since $f$ is integrable, we find that 
    \[
    \lim_{x\to\infty} \int_{x/2}^\infty f(t)\dt =0.
    \]
    The result follows by the sandwich criterion.
\end{proof}

We briefly note that it is impossible to quantify the convergence in Lemma \ref{lemma monotone integrable decay}. A stronger version of Example \ref{example no quantitative improvement} is given below in Example \ref{example no quantitative improvement l12}.

\begin{example}\label{example no quantitative improvement}
Assume that $\phi:[0,\infty)\to\R$ is a monotone increasing function such that $\lim_{t\to\infty}\phi(t) = +\infty$. We aim to show that there exists a monotone decreasing integrable function $g:[0,\infty)\to\R$ such that 
\[
\limsup_{t\to\infty} \big(t\cdot \phi(t)\cdot g(t)\big) = +\infty.
\]
In other words, we cannot guarantee that $g(t) \leq \frac{C}{t\,\phi(t)}$ for all large times for any given constant $C$. Of course, if $1/ (t\cdot \phi(t))$ fails to be integrable (e.g.\ for $\phi(t) = \log t$), then $\liminf_{t\to\infty}\big(t\cdot \phi(t)\cdot g(t)\big) =0$.

Let $R_n$ be a monotone increasing sequence of positive numbers such that the series 
\[
\sum_{n=1}^\infty \frac1{\sqrt{\phi(R_n)}} \quad\text{ converges. Then }
g(t) = \sum_{n=1}^\infty \frac1{R_n\,\sqrt{\phi(R_n)}}\,1_{(0, R_n]}(t)
\]
satisfies 
\[
\int_0^\infty g(t) \dt = \sum_{n=1}^\infty R_n \cdot \frac1{R_n\,\sqrt{\phi(R_n)}} = \sum_{n=1}^\infty \frac1{\sqrt{\phi(R_n)}} < +\infty
\]
and 
\begin{align*}
\limsup_{t\to\infty} \big(t\cdot \phi(t)\cdot g(t)\big) &\geq \limsup_{n\to\infty}\big(R_n \cdot \phi(R_n) \cdot g(R_n) \big) \\
    &\geq \limsup_{n\to\infty}\left(R_n\cdot \phi(R_n) \cdot \frac1{R_n\,\sqrt{\phi(R_n)}}1_{(0, R_n]}(R_n)\right) = \lim_{n\to\infty}\sqrt{\phi(R_n)} = +\infty.
\end{align*}
Furthermore, $g$ is a sum of decreasing functions and thus decreasing as well. It is easy to extend the example to functions which are continuous.
\end{example}

\section{Gradient flows in continuous time}\label{section gradient flow}
\subsection{Gradient flows for convex functions: general observations}

We present gradient flows in the context of finite-dimensional Euclidean spaces, but all arguments carry over directly to Hilbert spaces. For the sake of avoiding technical complications, we avoid thinking about infinite-dimensional spaces except for Section \ref{section Hilbert}, where they are needed for a counterexample. We note however that differences between gradient flows in finite-dimensional and infinite-dimensional Hilbert spaces are well-documented -- \cite{baillon1978exemple} gives an example of a gradient flow of a convex function in an infinite-dimensional Hilbert space which does not converge in the norm topology. A comparable guarantee for heavy ball optimization is given in \cite{attouch2018fast}.

Let $f:\R^m\to \R$ be a convex $C^1$-function and assume that $x$ solves the gradient-flow equation
$
\dot x_t = - \nabla f(x_t).
$
Then by construction, the energy dissipation identity
\[
\frac{d}{dt} f(x_t) = \nabla f(x_t)\cdot \dot x_t = - \|\nabla f(x_t)\|^2\leq 0
\]
holds. We review some additional well-known results.

\begin{lemma}
Let $x^*\in \R^m$. Then the function $L:[0,\infty) \to \R$,
\[
L(t) = t\,\big(f(x_t) - f(x^*)\big) + \frac12 \|x_t - x^*\|^2
\]
is non-increasing.
\end{lemma}

We refer to $L$ as the Lyapunov function associated with $x^*$.

\begin{proof}
Due to the first order convexity condition, we find that
\begin{align*}
L'(t) &= \big(f(x_t) - f(x^*)\big) + t\,\nabla f(x_t)\cdot \dot x_t + \langle x_t - x^*, \dot x_t\rangle\\
	&= f(x_t) + \langle \nabla f(x_t), \,x^*-x_t\rangle - f(x^*) - t\,\|\nabla f(x_t)\|^2\\
	&\leq  - t\,\|\nabla f(x_t)\|^2\\
	&\leq 0.
\end{align*}
\end{proof}

As an immediate corollary, we find that gradient flows are consistent in convex optimization, irrespective of whether the minimum is attained, or even finite.

\begin{corollary}
If $f$ is convex and $x_t$ a gradient flow of $f$, then $\lim_{t\to \infty} f(x_t) = \inf_{x\in\R^m} f(x)$.
\end{corollary}

\begin{proof}
Since $\frac{d}{dt} f(x_t) = \langle \nabla f(x_t), \:\dot x_t\rangle = - \|\nabla f(x_t)\|^2$, we find that $f(x_t)$ is monotone decreasing in time. In particular, the limit $\lim_{t\to\infty} f(x_t)$ exists (but may be $-\infty$ if $f$ is not bounded from below).

Assume for the sake of contradiction that $\lim_{t\to \infty} f(x_t) > \inf_{x\in\R^m} f(x)$. Choose $x^*\in \R^m$ such that $f(x^*) < \lim_{t\to\infty} f(x_t)$ and consider the associated function $L$. Then
\[
L(0) \geq L(t) \geq t\,\big(f(x_t)- f(x^*)\big) \geq \frac{\lim_{s\to\infty} f(x_s) - f(x^*)}2\, t 
\]
for all sufficiently large $t$. As the term on the right grows uncontrollably as $t\to\infty$, we have reached a contradiction.
\end{proof}

\subsection{Gradient flows for convex functions with minimizers: Hilbert spaces}\label{section Hilbert}

The assumption that $f$ has a minimizer has profound impact. In particular, if $f(x^*) = \inf_{x\in\R^m}f(x)$, the first term in $L$ is non-negative. We conclude that
\[
f(x_t) - \inf_{x\in\R^m}f(x) \leq \frac{L(t)}t \leq \frac{L(0)}t = \frac{\|x_0 - x^*\|^2}{2t},
\]
i.e.\ the excess objective $f(x_t) - \inf f$ decays at least as fast as $C/t$ for some $C>0$.
The question remains whether the rate of $1/t$ is optimal, and the fact that the energy dissipation $t\|\nabla f(x_t)\|^2$ is large suggests otherwise. We see that this is not the case -- unlike the upper bound $C/t$, the excess objective value $f(x_t) - \inf f$ is in fact {\em integrable} at infinity.

\begin{lemma}\label{lemma integrability gf}
Assume that $f$ is a convex function which has a minimizer $x^*$ and $x$ is a gradient flow of $f$. Then
\[
\int_0^\infty f(x_t) - \inf f\dt \leq \frac{\|x_0 - x^*\|^2}2.
\]
\end{lemma}

\begin{proof}
Note that 
\[
\frac{d}{dt}\|x_t - x^*\|^2 = 2\,\langle x_t -x ^*, \dot x_t\rangle = 2\,\langle \nabla f(x_t), x^*-x_t\rangle \leq 2\big\{f(x^*) - f(x_t)\big\} \leq 0.
\]
In particular, the function $\|x_t - x^*\|^2$ is monotone decreasing and bounded from below and thus has a limit. We find that
\[
\frac{\|x^* - x_0\|^2}2 \geq \frac{\|x^* - x_0\|^2}2 - \frac{\|x^* - x_T\|^2}2 = \int_0^T f(x_t) - f(x^*)\dt
\]
for all $T>0$. The result follows by taking $T\to\infty$.
\end{proof}

\begin{proof}[Alternative Proof of Lemma \ref{lemma integrability gf}]
    Again, let $L$ be the Lyapunov function associated to a minimizer $x^*$. Then we compute that 
\begin{align*}
\int_0^\infty f(x_t) - \inf f\dt &= - \int_0^\infty \int_t^{\infty} \frac{d}{ds}f(x_s)\ds \dt
	= \int_0^\infty \int_t^{\infty} \|\nabla f(x_s)\|^2\ds \dt\\
	&= \int_0^\infty \int_0^{s} \|\nabla f(x_s)\|^2\dt \ds
	=\int_0^\infty s\,\|\nabla f(x_s)\|^2\ds
	\leq -\int_0^\infty L'(s)\ds
	\leq L(0).
\end{align*}
    It is possible to exchange the order of integration here due to a Theorem of Tonelli, see e.g.\ \cite[Kapitel 8.5]{konigsberger2013analysis}.
\end{proof}

Since $f(x_t) - \inf f$ is additionally decreasing, Lemmas \ref{lemma monotone integrable decay} and \ref{lemma integrability gf} imply the following.

\begin{corollary}\label{corollary faster convergence gf}
Assume that $f$ is a convex function which has a minimizer and $x$ is a gradient flow of $f$. Then
\[
\lim_{t\to\infty}t\cdot\big( f(x_t) - \inf_{x\in\R^m}f(x) \big)=0.
\]
\end{corollary}

Additionally, we can compare $f(x_t)$ to a function which is non-integrable at infinity such as $1/(t\,\log t)$, we immediately obtain the qualitative statement that
\[
\liminf_{t\to\infty}\big(t\cdot \log t\cdot \big(f(x_t) - \inf f\big)\big) = 0.
\]
Stronger statements are available, but harder to formulate \cite{niculescu2011note}.
To illustrate that the improvement from integrability can be made quantitative, we obtain a non-asymptotic risk bound for the optimal iterate in a given range. While uncommon in convex optimization, such `optimal iterate' bounds are the norm in non-convex optimization.

\begin{lemma}
For any $t>1$, we have
\[
\min_{t\leq s\leq t\log t} \big(s\cdot \big(f(x_s) - \inf f\big)\big) \leq \frac{\|x_0-x^*\|^2}{2\,\log(\log t)}.
\]
\end{lemma}

\begin{proof}
    For simplicity, $\inf f=0$. If we have $f(x_s)\geq \frac{\eps}{s\,\log (\log t)}$ for some $\eps>0$ and $t\leq s \leq t\log t$, then
\[
    \frac{\|x_0-x^*\|^2}2 \geq \int_0^\infty f(x_s)\ds
        %&\geq\int_t^{t\,\log t} f(x_s)\ds\\
        %\geq\int_t^{t\,\log t} \frac{\eps}{s\,\log(\log s)}\ds
        \geq \frac{\eps}{\log(\log t)} \int_t^{t\log t} \frac1s\ds
        %= \frac{\eps}{\log(\log t)} \log\left(\frac{t\,\log t}{t}\right)
        = \eps.\qedhere
\]
\end{proof}

In other words, $f(x_t) - \inf f$ decays slightly faster than $O(1/t)$ in a way that can be made precise. 
% \subsection{Gradient flows for convex functions with minimizers: Lower bounds}
We now demonstrate that the characterization of Lemma \ref{lemma integrability gf} is sharp, at least in infinite-dimensional Hilbert spaces. The existence and regularity of a gradient flow curve in a Hilbert space follows by the Picard-Lindel\"off theorem in the infinite-dimensional case just as it does in the finite-dimensional situation.

%\todo[inline]{probably worth a reference}

Note that by `gradient flows in Hilbert spaces' we refer to gradient flows of continuous convex functionals defined on a Hilbert space. This does {\em not} cover PDEs which arise as gradient flows of convex functionals such as the Dirichlet energy which are only defined on a dense subset. Much greater care must be taken in that context to prove existence and interpret gradients.

\begin{lemma}\label{lemma slow convergence}
Let $H$ be a separable Hilbert space. Then there exists a quadratic convex function such that the following is true: For any monotone decreasing integrable function $g:[0,\infty)\to\R$, there exists a gradient flow solution $u(t)$ of $F$ such that
\[
F(u) = 0 \quad\LRa\quad u=0 \qquad\text{and}\qquad F(u(t)) \geq g(t)\quad\forall\ t\geq 1.
\]
\end{lemma}

\begin{proof}
Without loss of generality, we may consider $H = L^2(1,\infty)$ by isometry. Consider
\[
F:H\to \R, \qquad F(u) = \frac12\int_1^\infty \frac{u^2(s)}s \ds.
\]
Since $F$ is a continuous quadratic form, it is Frechet differentiable with gradient $\big(\nabla F(u)\big)(s) = \frac{u(s)}s$.
The gradient flow of $F$ acts pointwise in $s$: $u(t,s) = e^{-t/s}u_0(s)$, so if $t\geq 1$, then
\[
F(u(t)) = \frac12 \int_1^\infty\frac{u_0^2(s)}s \,e^{-2t/s}\ds \geq 
	\frac12 \int_t^\infty\frac{u_0^2(s)}s \,e^{-2t/s}\ds \geq \frac1{2e^2} \int_t^\infty\frac{u_0^2(s)}s \ds.
\]
In order to identify the lower bound, we require
\[
g(t) = \frac1{2e^2} \int_t^\infty \frac{u_0^2(s)}s\ds \qquad\LRa\quad g'(t) = - \frac{u_0^2(t)}{2e^2\,t}\text{ and } \lim_{t\to \infty} g(t) = 0.
\]
We therefore select $u_0(s):= \sqrt{-2e^2s\,g'(s)}$ and verify that 
\[
\frac1{2e^2}\int_1^R u_0^2(s)\ds = -\int_1^R s\,g'(s)\ds = g(1) - Rg(R) + \int_1^R g(s)\ds.
\]
Since $g$ is integrable and monotone decreasing, we have $\limsup_{R\to\infty}Rg(R) = 0$ by Lemma \ref{lemma monotone integrable decay}.
% \[
% \limsup_{R\to\infty}Rg(R) 
%     = 2\limsup_{R\to\infty}\int_{R/2}g(R)\ds 
%     \leq 2\limsup_{R\to\infty} \int_{R/2}g(s)\ds = 2\lim_{R\to\infty}\int_{R/2}^\infty g(s)\ds = 0.
% \]
Thus
\[
\frac1{e^2}\int_1^R u_0^2(s)\ds = g(1) + \int_1^R g(s)\ds < \infty.
\]
In particular, $u_0\in H$ is a valid initial condition. 
\end{proof}

Notably, the convergence of $t\big(f(x_t)\to \inf f\big)$ to zero can be arbitrarily slow, even for a fixed quadratic functional on an infinite-dimensional Hilbert space, depending only on the initial condition. This quadratic form is `infinitely flat' by its minimum: If $\phi:(0,\infty)\to(0,\infty)$ is any monotone increasing function such that $\lim_{r\to0} \phi(r) = 0$, then there exists a sequence $u_n\in H$ such that
\[
\lim_{n\to \infty} u_n = 0, \qquad \lim_{n\to \infty}\frac{F(u_n)}{\phi(\|u_n\|)} = 0.
\]
In our example,
such a sequence is given for example by
\[
u_n = \frac1n \,1_{\{R_n < s < 1+R_n\}}, \qquad R_n = \frac1{\phi(1/n)}\qquad\text{since } \|u_n\|_{L^2} = 1/n, \qquad F(u_n)\leq \frac1{n\,R_n}.
\]

% We now demonstrate that, if there exists a slowly decaying gradient flow for a convex function with minimizer
% \begin{enumerate}
% \item on a finite dimensional space or
% \item on a Hilbert space, and the gradient flow curve has finite length, then
% \end{enumerate}
% on a Hilbert space, then there exists an equally slowly decaying gradient flow for a convex function on $\R$ which possesses a minimizer.

\subsection{Gradient flows for convex functions with minimizers: Real line}

We can analyze the one-dimensional case more directly. Let $x_t$ be a gradient flow curve for a $C^2$-smooth convex function $f:\R^m\to\R$. Define $g(t) = f(x_t)$. Then 
\[
g'(t) = - \big\|\nabla f(x_t)\big\|^2, \qquad
g''(t) = -\frac{d}{dt} \big\|\nabla f(x_t)\big\|^2 = - 2\,\nabla f\cdot (D^2f)\dot x = 2\,\nabla f\cdot (D^2f)\nabla f  \geq 0,
\]
i.e.\ $g$ is $C^2$-smooth, strictly monotone decreasing and convex. 
Focusing on the one-dimensional case, the gradient flow has a limit since $x_t$ is either monotone increasing or decreasing. To see this, note that $\dot x_t = -f'(x_t)$ cannot change sign along the gradient flow without passing through a minimizer, at which point the trajectory stops moving.
% If $f$ has a minimizer $x^*$, then there exists $x_\infty$ such that $x_t\to x_\infty$ as $t\to\infty$, and we observe furthermore that
% \[
% |x_0 -x_\infty| \leq \int_0^\infty |\dot x_t| \dt = \int_0^\infty |f'(x_t)| \dt = \int_0^\infty \sqrt{- g'(t)}\dt
% \]
% since $x$ cannot go `back and forth' in one dimension, i.e.\ it moves left or right on the real line, but does not change direction along its trajectory. 
More generally, the gradient flow curves of a convex function (with minimizers) on a finite-dimensional space have finite length \cite{convex_finite_length, gupta2021path}, so we find that 
\[
 \int_0^\infty \sqrt{- g'(t)}\dt = \int_0^\infty \|\nabla f(x_t)\|\dt = \int_0^\infty \|\dot x_t\|\dt < \infty.
\]
We see that this in fact characterizes the energy decay in gradient flows completely.

\begin{lemma}\label{lemma convex decay 1d}
Let $g:[0,\infty)\to [0,\infty)$ a monotone decreasing convex $C^2$-function such that 
\[ 
|g'(0)| + \int_0^\infty \sqrt{-g'(t)}\dt < +\infty.
\]
Then there exist
\begin{enumerate}
\item a convex function $\phi:\R\to[0,\infty)$ such that $\phi(x) = 0$ if and only if $x=0$ and
\item a gradient flow $x_t$ of $\phi$ such that $f(x_t) = g(t)$ for all $t\in\R$.
\end{enumerate}
\end{lemma}

We note that $\lim_{t\to\infty}g(t)$ exists since $g$ is monotone decreasing. Without loss of generality, we assume that $\lim_{t\to\infty}g(t)=0$. We show that the conditions of Lemma \ref{lemma convex decay 1d} recover two previous characterizations, at least in part.

\begin{enumerate}
\item First, we note that for a convex decreasing function $g$ we have
    \begin{align*}
        g(t) &= \int_t^\infty - g'(s)\ds \leq \sqrt{-g'(t)}\int_t^\infty \sqrt{-g'(s)}\ds \leq \frac1t \int_0^t\sqrt{-g'(s)}\ds\cdot \int_t^\infty \sqrt{-g'(s)}\ds\\
        &\leq \frac{\left(\int_0^\infty \sqrt{-g'(s)}\ds\right)^2}{4t}.
    \end{align*}
    In the setting of gradient flows, $\int_0^\infty \sqrt{-g'(t)}\dt = \int_0^\infty \|\nabla f(x_t)\|\dt$ is the length of the gradient flow curve. In particular, this replaces the estimate
    \[
    f(x_t) - \inf f\leq \frac{\|x_0-x^*\|^2}{2t}
    \]
    with a better constant $1/4$ in place of $1/2$, but with the length of the trajectory rather than the Euclidean distance of its endpoints. This can vastly overestimate the true constant, but without access to the original geometry, it is a valid replacement. In one dimension, it is a strict improvement.
    
    \item Next, we show that $\int_0^\infty g(t)\dt<\infty$. Namely, since $\sqrt{-g'}$ is monotone decreasing, we have
    \[
    0\leq g(t) = -\int_t^\infty g'(s)\ds \leq \sqrt{-g'(t)} \int_t^\infty \sqrt{-g'(s)}\ds \leq \sqrt{-g'(t)}\left(\int_0^\infty\sqrt{-g'(s)}\ds\right),
    \]
    so  $g \in L^1(0,\infty)$.
    \end{enumerate}
    
\begin{proof}[Proof of Lemma \ref{lemma convex decay 1d}]
{\bf Setup.}
Since $g$ is monotone, convex and integrable, we note that $g'(t)<0$ or $g(s)\equiv 0$ for all $s\geq t$. The second case is a simpler variation, so we may assume that  $g'(t)\neq 0$ for any $t\in(0,\infty)$.
Thus the map
\[
\Psi :[0,\infty)\to [0,\infty), \qquad \Psi(t) =  \int_t^\infty \sqrt{-g'(s)}\ds 
\]
is $C^1$-smooth and strictly monotone decreasing. In particular, we can define $\phi$ on the interval $\left[0, X\right]$ by
\[
\phi\left(\int_t^\infty \sqrt{-g'(s)}\ds\right) = g\left( t \right)\qquad\text{where}\quad X = \int_t^\infty \sqrt{-g'(s)}\ds.
\]
It is easy to extend $\phi$ in a $C^1$-fashion as as $\phi(x) = x^2$ for $x<0$ and $\phi(x) = \phi(X) + \phi'(X)\,(x-X)$ for $x>X$. For the finiteness of derivatives, see the next step.

{\bf Convexity of $\phi$.}
We can easily compute the derivatives as
\begin{align*}
g'(t) &= -\phi'\left(\int_t^\infty \sqrt{-g'(s)}\ds\right) \,\sqrt{-g'(t)}.
\end{align*}
Note in particular that $\phi'(X) = \sqrt{-g'(0)}$ and $\phi'(0) = \lim_{t\to\infty}\sqrt{-g'(t)} = 0$, since $\sqrt{-g'}$ is integrable and monotone decreasing (since $g$ is convex). We compute further that 
\begin{align*}
g''(t) &= \phi''\left(\int_t^\infty \sqrt{-g'(s)}\ds\right) \,\sqrt{-g'(t)}^2 - \phi'\left(\int_0^t \sqrt{-g'(s)}\ds\right) \,\frac{d}{dt}\sqrt{-g'(t)}\\
	&= - \phi''\,g' + \phi'\,\frac{g''}{2\sqrt{-g'}}\\
	&= - \phi''\,g' - {\sqrt{-g'}}\,\frac{g''}{2\sqrt{-g'}}\\
	%&= - \phi''\,g' + \frac{g'g''}{2g'}\\
	&= - \phi''\,g' + g''/2,
\end{align*}
so
\[
\phi''\left(\int_0^t \sqrt{-g'(s)}\ds\right) = -\frac{g''}{2g'} \geq 0.
\]
If $g$ is (strictly) monotone decreasing and convex, we see that $\phi$ is convex as well.

{\bf Gradient flow of $\phi$.} Consider the gradient flow $s_t$ of $\phi$, i.e.\ the solution of the ODE 
\[
\dot s_t = - \phi'(s_t) = - \sqrt{-g'(\Psi^{-1}(s_t))}.
\]
Note that $S_t:= \Psi(t)$ satisfies
\[
\dot S_t = \Psi'(t) = - \sqrt{-g'(t)} = - \sqrt{-g' \big(\Psi^{-1}\big(\Psi(t)\big)\big)} = - \sqrt{-g'\big(\Psi^{-1}(S_t)\big)},
\]
i.e.\ $s_t$ and $S_t$ solve the same differential equation. We conclude that $s_t = \Psi^{-1}(t)$. 
 As usual, we have
\[
\frac{d}{dt} \phi(s_t) = \phi'(s_t)\,\dot s_t =  \big(\phi'(s_t)\big)^2 = g'\big(\Psi^{-1}(s_t)\big) = -g'(t).
\]
Again, we conclude that $\phi(s_t) = g(t)$.
\end{proof}

Based on Lemmas \ref{lemma slow convergence} and \ref{lemma convex decay 1d}, we demonstrate that there is a fundamental difference between the gradient flows of convex functions in finite-dimensional and infinite-dimensional Hilbert spaces. More precisely, while the integrability of $\sqrt{-g'}$ implies the integrability of $g$, the two are not equivalent:

Consider the function $g_\alpha(t) = \frac1{t\,(\log t)^\alpha}$. Then
\[
\int_2^\infty g_\alpha(t) = \frac{\log(t)^{1-\alpha}}{1-\alpha} \bigg|_{t=2}^{t\to\infty} = \frac{(\log 2)^{1-\alpha}}{1-\alpha} < +\infty,
\]
for $\alpha>1$ but 
\[
g_\alpha'(t) = - \left(\frac{n}{x^2\,(\log x)^{1+\alpha}} + \frac{1}{x^2\,(\log x)^\alpha}\right) \leq - \frac{1}{x^2\,(\log x)^\alpha}.
\]
In particular
\[
\int_2^\infty \sqrt{-g_\alpha'(t)}\dt \geq \int_2^\infty \frac1{t\,(\log t)^{\alpha/2}}\dt = +\infty
\]
if $\alpha\leq 2$ since
\[
\int_2^\infty \frac1{t\,\log t}\dt = \log(\log t)\bigg|_{t=2}^{t=\infty} = + \infty
\]
and $\log$ is monotone increasing, $\alpha/2>0$. Thus $g_\alpha$ is not the decay function for the gradient flow of a convex function in finite dimension for $\alpha\in(1,2]$. We want to conclude that
\begin{enumerate}
    \item $\liminf_{t\to \infty} \big(t\,(\log t)^2\cdot f(x_t)\big) =0$ for the gradient flow of a convex function in finite dimension.
    \item There is a qualitatively different condition on the decay rate in finite dimension compared to Hilbert spaces.
\end{enumerate}
The question is: Is there a convex function $\tilde g_\alpha$ such that
\[
\liminf_{t\to \infty}\frac{\tilde g_\alpha(t)} {g_\alpha(t)}\geq 1, \qquad \lim_{t\to\infty}\tilde g_\alpha(t) = 0\qquad\text{and }\int_2^\infty \sqrt{-\tilde g_\alpha'(t)}\dt < \infty?
\]
If such a function $\tilde g_\alpha$ exists, then $g_\alpha$ itself may not arise as the decay curve of a gradient flow, but it does not serve as an lower barrier, even asymptotically. We prove that this is not possible in Appendix \ref{apppendix l12 derivative}. Namely, we prove the following auxiliary statement.

\begin{restatable*}{lemma}{lemmaa}\label{square-root-integral-lower-bound}
Let $g, G:[0,\infty) \to [0,\infty)$ be decreasing, differentiable convex functions such that
\[
\lim_{t\to\infty} G(t) = \lim_{t\to\infty} g(t) = 0, \qquad \liminf_{t\to\infty} \frac{G(t)}{g(t)} > 0.
\]
Then 
\[
\int_1^\infty \sqrt{-g'(t)}\dt = +\infty \qquad\Ra \qquad  \int_1^\infty \sqrt{-G'(t)}\dt = +\infty.
\]
\end{restatable*}

To understand Lemma \ref{square-root-integral-lower-bound}, consider a simpler task first: Minimize $\int_0^\infty \sqrt{|G'(t)|}\dt$ in the class of functions $G$ such that $G(0) = 1$ and $\lim_{t\to\infty}G(t) = 0$. This problem is not well-posed as the function $G_r(t) =\max\{1-rt, 0\}$ achieves
\[
\int_0^\infty \sqrt{|G_r'(t)|}\dt = \int_0^{1/r}\sqrt{r}\dt = \frac1{\sqrt r}, \qquad r>0.
\]
As $r\to \infty$, the integral approaches zero, i.e.\ the energy infimum is $0$ and is not attained. This is due to the fact that for the square root of the derivative, short and steep segments are heavily discounted. 

The function $G_r$ is convex for all $r>0$. An extension of the argument above could be used to construct $G\geq g$ such that $\|G'\|_{L^{1/2}(0,\infty)}$ is arbitrarily small by introducing many short, steep segments and keepoing $G$ mostly constant away from these fast transitions. However, in combination, the constraints that $G$ must be convex and (a version of) $G\geq g$ induce a non-trivial competition: $G$ should be as steep as possible, since large derivatives on short segments are heavily discounted. However, it cannot concentrate steep segments in many places since its derivative is a monotone function.

The proof of Lemma \ref{square-root-integral-lower-bound} is the most technically challenging part of the article. It primarily uses the concavity of the function $z\mapsto \sqrt{z}$ and the statement remains valid for more general concave functions. We postpone the proof to the Appendix in order to focus on the application to gradient flows for now.

In particular, we have shown the following.

\begin{corollary}\label{corollary log2 decay}
    Let $f:\R\to\R$ be a convex $C^2$-function and $x^*\in\R$ such that $f(x^*) = \inf f$. Then it is not possible that $f(x_t) \geq \frac \eps{t\,\log^2t}$ for a fixed $\eps>0$ and all large $t$.
\end{corollary}

We deduce that
\begin{equation}\label{eq liminf in 1d is better}
\liminf_{t\to\infty} \big(t\,\log^2t \cdot \big(f(x_t) - \inf f\big)\big) = 0.
\end{equation}
We note however that a substantial improvement over the rate $O(1/t)$ is not possible, and in fact our argument shows that $f(x_t) - \inf f$ may satisfy 
\[
\limsup_{t\to\infty} \big(t\cdot \log^\alpha(t) \cdot \big(f(x_t)-\inf f\big)\big) = +\infty 
\]
for any $\alpha>2$, even in one dimension.
It is tempting, but unfortunately incorrect to assume that the lower limit in \eqref{eq liminf in 1d is better} could be replaced by an proper limit for functions which satisfy the hypothesis that $\int_0^\infty \sqrt{-g'(t)}\dt < +\infty$. We extend Example \ref{example no quantitative improvement} to this scenario and show that no stronger version of Lemma \ref{lemma monotone integrable decay} can be achieved, even under the stronger condition of finite path-length.

\begin{example}\label{example no quantitative improvement l12}
    Let $\phi:(0,\infty)\to(0,\infty)$ be a monotone increasing function such that $\lim_{t\to\infty}\phi(t) = +\infty$. We will show that there exists a convex function $g:(0,\infty)\to(0,\infty)$ such that
    \[
    \lim_{t\to\infty} g(t) = 0, \qquad \int_0^\infty\sqrt{-g'(t)}\dt < +\infty, \qquad \limsup_{t\to\infty} \big(t\cdot\phi(t) \cdot g(t) \big) = +\infty.
    \]
    Let $R_n$ be a sequence such that $\sum_{n=1}^\infty \frac1{\sqrt[3]{\phi(R_n)}}<\infty$. Define 
    \[
    \sqrt{-g'(t)} = \sum_{n=1}^\infty \frac1{R_n\,\sqrt[3]{\phi(R_n)}} \,1_{(0,2R_n]}(t).
    \]
    Then 
    \begin{enumerate}
        \item $\sqrt{-g'}$ is monotone decreasing, i.e.\ $g'$ is monotone increasing, i.e.\ $g$ is convex. 
        \item $\sqrt{-g'}$ is integrable by the same argument as in Example \ref{example no quantitative improvement}.
    \end{enumerate}
    We note that 
    \begin{align*}
    g(t) &= \int_t^\infty - g'(t) \dt \geq \int_t^\infty \sum_{n=1}^\infty \frac{1}{R_n^2\,\big(\phi(R_n)\big)^{2/3}}\,1_{(0,2R_n]}(t)\dt 
    \end{align*}
    Then in particular
    \begin{align*}
    \limsup_{t\to\infty} (t\cdot \phi(t)\cdot g(t)\big) &\geq \limsup_{n\to\infty} \left(R_n\cdot \phi\left(R_n\right) \cdot g(R_n)\right) \\
    &\geq \limsup_{n\to\infty} R_n\cdot \phi(R_n)\int_{R_n}^{2R_n} \frac{1}{R_n^2\,\big(\phi(R_n)\big)^{2/3}}\,1_{(0,2R_n]}(t)\dt\\
        &= \limsup_{n\to\infty} \left(R_n \cdot \phi(R_n) \cdot \frac{1}{R_n\,\big(\phi(R_n)\big)^{2/3}}\right)\\
        &= \lim_{n\to\infty} \phi(R_n)^{1/3} = +\infty.
    \end{align*}
    Again, it is easy to generalize the example to a version where $g$ is infinitely smooth.
\end{example}

\subsection{Gradient flows for convex functions with minmizers: Finite dimension}

In this section, we show that there is no  difference between the decay rates which can be guaranteed for convex functions on finite-dimensional spaces and convex functions on the real line. We exploit two facts: 
\begin{enumerate} 
\item Only the geometry of the objective function along the gradient direction matters to the gradient flow, i.e.\ we can consider the objective function only along the curve traced by the gradient flow itself (in a suitable reparametrization). 

\item A gradient flow curve in a finite-dimensional space always has finite length \cite{convex_finite_length}. More generally, gradient flow curves are completely characterized by the `self-contracting' property that
\[
t_1< t_2<t_3 \quad\Ra\quad \|\gamma(t_2) - \gamma(t_3)\| \leq \|\gamma(t_1) - \gamma(t_3)\|
\]
in the Euclidean norm \cite{daniilidis2010asymptotic, durand2019self}. Self-contracting curves were shown to have finite length in fairly general circumstances \cite{stepanov2017self}.
\end{enumerate}
We note, however, that even gradient flows in two dimensions can be surprisingly complicated. In \cite{daniilidis2010asymptotic}, the authors construct a gradient flow which winds around the unique minimizer of a convex function infinitely often. Such a construction is even possible for a function which is analytic except at the minimizer \cite{daniilidis2022convex}.

We note that these results do not hold in infinite-dimensional Hilbert spaces. In \cite{baillon1978exemple}, the authors construct the gradient flow of a convex function on a Hilbert space which does not converge to a minimizer in the norm topology. In particular, as it does not converge to a limit, the gradient flow curve has infinite length.

\begin{lemma}
Assume that $H$ is a Hilbert space, $f:H\to\R$ is a convex function with a Lipschitz-continuous gradient and $x_t$ is a gradient flow of $f$. If $x$ has finite length, then there exist
\begin{enumerate}
    \item convex $C^1$-function $g:\R\to\R$ which has a minimizer and 
    \item a gradient flow $s_t$ of $g$ such that
\end{enumerate}
\[
f(x_t) = g(s_t) \qquad \forall\ t>0.
\]
\end{lemma}

In particular, the statement applies to all gradient flow lines in finite-dimensional Hilbert spaces. The Lemma remains true without the finite length assumption, but becomes somewhat less instructive as $g$ does not have a minimizer in this case.

\begin{proof}
{\bf Setup.} Let us compare two curves: The solution of the gradient flow equation $\dot x_t = -\nabla f(x_t)$ or the time-normalized gradient flow $\dot z(t) = - \frac{\nabla f(z(t))}{\|\nabla f(z(t))\|}$. Then
\[
x_t = z(\phi(t)) \qquad\text{where}\quad \phi(t) = \int_0^t \|\nabla f(x_s)\|\ds
\]
To see this, consider $y(s) = x(\phi^{-1}(s))$ and note that
\[
\frac{d}{ds} y(s) = \dot x(\phi(s))\,\big(\phi^{-1}\big)'(s) = -\nabla f\big(x(\phi^{-1}(s))\big) \,\frac{1}{\phi'(\phi^{-1}(s))} = \frac{-\nabla f(x(\phi^{-1}(s)))}{\| \nabla f(x(\phi^{-1}(s))) \|} = - \frac{\nabla f (y(s))}{\|\nabla f(y(s))\|}.
\]
Since $y$ and $z$ solve the same ODE and $\nabla f/\|\nabla f\|$ is locally Lipschitz-continuous on the set where $\nabla f\neq 0$, we find by the uniqueness contribution of the Picard-Lindel\"off theorem that $y\equiv z$. 

Recall that the assumption of finite length means that 
\[
\lim_{t\to\infty} \int_0^t \|\dot x_s\|\ds = \lim_{t\to\infty} \int_0^t \|\nabla f(x_s)\|\ds = \lim_{t\to\infty} \phi(t) < \infty.
\]
We denote $R:= \lim_{t\to\infty}\phi(t)$.

{\bf Step 1.} In this step, we construct $g:[0,R]\to \R$ as $g(s) = f(z(s))$. Then
\begin{align*}
g'(s) &= \nabla f(z(s))\cdot \dot z(s) = - \|\nabla f(z(s))\|\\
 g''(s) &= - \frac{\nabla f(z(s))}{\|\nabla f(z(s))\|} \cdot D^2f(z(s)) \cdot \dot z(s)
 	= 2 \frac{\nabla f(z(s))}{\|\nabla f(z(s))\|} \cdot D^2f(z(s)) \cdot\frac{\nabla f(z(s))}{\|\nabla f(z(s))\|} \geq 0
\end{align*}
since $D^2f$ is non-negative semi-definite. The function $g$ is therefore convex and $C^1$-smooth on its domain of definition.

We extend $g$ to the entire real line by setting $g(s) = \inf f$ if $s>R$ and $g(s) = g(0) + g'(0) s$ if $s<0$. This results in a $C^1$-extension, but we note that it could easily be made $C^2$-smooth, at least at $s=0$.

{\bf Step 2.} We consider a gradient flow curve $s$ of $g:\R\to\R$ such that $s(0) = 0$. Then $g(s(0)) = f(x_0)$ by construction and
\[
\dot s_t = -g'(t) = - \|\nabla f(z(s))\| \quad\Ra\quad \frac{d}{dt} g(s_t) = - |g'(s_t)|^2 = - \|\nabla f(z(s))\|^2.
\]
In particular $f(x_t) = g(s_t)$ for all $t$ since their derivatives coincide and they take the same value at $t=0$.  
\end{proof}

\subsection{Gradient flows for convex functions without minimizers}
\label{section no minimizers}

For convex functions without minimizers, the decay of energy along a gradient flow can be arbitrarily slow, even in one dimension.

\begin{lemma}\label{lemma gf no minimizers}
    Assume that $g:[0,\infty)\to \R$ is a function such that $\lim_{t\to\infty}g(t) = 0$. Then there exists a convex function $f:\R\to\R$ and a gradient flow $x_t$ of $g$ such that $\inf_{x\in\R}f(x) =0$ and $f(x_t)\geq g(t)$ for all $t\geq 1$.
\end{lemma}

\begin{proof}
    {\bf Step 1.} We make two adjustments.
    \begin{enumerate}
        \item Replace $g$ by $\tilde g(t) = \max_{s\geq t}g(s)$ to ensure that the function is monotone non-increasing.
        \item Replace $g(t)$ by $\int_{t-1}^t g(s)\ds \geq g(t)$ (since $g$ is non-increasing).
    \end{enumerate}
    Using the two modifications, we may assume that $g$ is $C^1$-smooth and monotone non-increasing. 
    
    {\bf Step 2.} We construct a {\em convex} function $\phi\geq g$ such that $\lim_{t\to \infty}\phi(t) = 0$. Namely, set
    \[
    \phi(t) = \int_t^\infty (s-t)\,\frac{-g'(s)}s\ds.
    \]
    Then $ \lim_{t\to \infty} \phi(t) = 0$ since $g'$ is integrable and $\frac{s-t}s\leq 1$. Furthermore
    \begin{align*}
        \phi'(t) &= \int_t^\infty \frac{-g'(s)}s\ds 
    \end{align*}
    is monotone decreasing since the domain of integration is shrinking and $g'\leq 0$. Thus $g$ is convex. Finally, we note that
    \[
    \phi(t) = \int_t^\infty (s-t)\,\frac{-g'(s)}s\ds \leq \int_t^\infty \big(-g'(s)\big)\ds = g(t)
    \]
    for all $t>0$ since $-g'\geq 0$ and $\frac{s-t}{s}\leq 1$. On the other hand
    \[
    \phi(t) = \int_t^\infty \frac{s-t}s\,\big(-g'(s)\big)\ds 
    %\geq \int_{2t}^\infty \frac{s-t}s\,\big(-g'(s)\big)\ds 
    \geq \int_{2t}^\infty \big(1-t/s\big)\,\big(-g'(s)\big)\ds
    \geq \frac12 \int_{2t}^\infty\big(-g'(s)\big)\ds = \frac{g(2t)}2.
    \]
    Rescaling $g$ before, we can reach $\phi(t) \geq g(t)$ instead.

    {\bf Step 3.} By the same argument as in Lemma \ref{lemma convex decay 1d}, we see that there exists a convex function 
    \[
    f:\R\to\R, \qquad f\left(\int_0^t \sqrt{\phi'(s)}\ds\right) = \phi(t)
    \]
    and a gradient flow $x_t$ of $f$ such that $f(x_t) = \phi(t) \geq g(t)$. The integrability of $\sqrt{\phi'}$ is not needed in this context as we do not insist that $f$ has a minimizer.
\end{proof}

\section{Gradient descent in discrete time} \label{section gradient descent}

In this section, we prove that the improved convergence result of Lemma \ref{lemma integrability gf} carries over to the explicit Euler time-stepping scheme for the gradient flow ODE, i.e.\ the gradient descent (GD) algorithm.

\subsection{Deterministic gradient descent}

We prove an analogue of Lemma \ref{lemma integrability gf} for gradient descent in discrete time.

\begin{lemma}
Let $f:\R^m\to\R$ be a convex function, $x^*\in\R^m$ such that $f(x^*) = \inf_{x\in\R^m}f(x)$ and $x_1\in\R^m$ an initial condition. Assume that $\nabla f$ is $L$-Lipschitz continuous with respect to the Euclidean norm. If the sequence $(x_n)_{n\in\N}$ follows the gradient descent law $x_{n+1} = x_n - \eta \nabla f(x_n)$ for a step-size $0<\eta<2/L$, then
\[
\eta\sum_{n=0}^\infty \big(f(x_n) - f(x^*)\big) \leq \frac{\|x_0 - x^*\|^2}2 + \frac{\eta}{2(1- L\eta/2)}\,\big(f(x_0) - f(x^*)\big).
\]
\end{lemma}

If $\eta = 1/L$, then this becomes
\[
\sum_{n=0}^\infty \big(f(x_n) - f(x^*)\big) \leq \frac{L\,\|x_0 - x^*\|^2}2 + \big(f(x_0) - f(x^*)\big).
\]

\begin{proof}
{\bf Step 1.} Denote $g_n = \nabla f(x_n)$. The discrete time analogue of the continuous time energy dissipation identity $\frac{d}{dt} f(x_t) = -\|\nabla f(x_t)\|^2$ is 
\[
f(x_{n+1}) \leq f(x_n) - \left(1-\frac{L\eta}2\right)\eta\,\|\nabla f(x_n)\|^2,
\]
as proved e.g.\ in \cite[Lemma 3.4]{bubeck2015convex}. If $\eta < 2/L$, the factor $\gamma= \left(1-\frac{L\eta}2\right)\eta$ is non-negative.

{\bf Step 2.}
The sequence $L_n = \|x_n - x^*\|^2$ satisfies
\begin{align*}
L_{n+1} &= \|x_{n+1} - x_n + x_n - x^*\|^2\\
    &= \|\eta \nabla f(x_n)\|^2 + 2\langle -\eta \nabla f(x_n), \,x_n-x^*\rangle + \|x_n-x^*\|^2\\
    &\leq \eta^2\,\|\nabla f(x_n)\|^2 - 2\eta \big(f(x_n) - f(x^*)\big) +L_n
\end{align*}
by the first order convexity condition. In particular, we have
\begin{align*}
\|x_0 - x^*\|^2 = L_0 &\geq \sum_{n=0}^\infty \big(L_n - L_{n+1}\big)\\
    &\geq 2\eta\sum_{n=0}^\infty \big(f(x_n) - f(x^*)\big) - \frac{\eta^2}\gamma \sum_{n=0}^\infty \gamma \,\|\nabla f(x_n)\|^2\\
    &\geq 2\eta\sum_{n=0}^\infty \big(f(x_n) - f(x^*)\big) - \frac{\eta^2}\gamma \sum_{n=0}^\infty \big(f(x_n) - f(x_{n+1})\big),
\end{align*}
so
\[
\eta\sum_{n=0}^\infty \big(f(x_n) - f(x^*)\big) \leq \frac{\|x_0 - x^*\|^2}2 + \frac{\eta^2}{2\gamma}\,\big(f(x_0) - f(x^*)\big).\qedhere
\]
\end{proof}

A discrete time analogue of Lemma \ref{lemma monotone integrable decay} with essentially the same proof states that, if $a_n$ is a monotone decreasing and summable sequence, then $\lim_{n\to\infty} (n\cdot a_n) = 0$.  As a consequence, we note the following.

\begin{corollary}
    Let $f:\R^m\to\R$ be a convex function, $x^*\in\R^m$ such that $f(x^*) = \inf_{x\in\R^m}f(x)$ and $x_1\in\R^m$ an initial condition. Assume that $\nabla f$ is $L$-Lipschitz continuous with respect to the Euclidean norm. If the sequence $(x_n)_{n\in\N}$ follows the gradient descent law $x_{n+1} = x_n - \eta \nabla f(x_n)$ for a step-size $\eta<2/L$, then
    \begin{enumerate}
        \item $\lim_{n\to \infty} n\cdot \big(f(x_n) - f(x^*)\big) = 0$.
        \item $\liminf_{n\to\infty} n\,\log(n) \cdot \big(f(x_n) - f(x^*)\big) =0$.
    \end{enumerate}
\end{corollary}

We conjecture that also in the discrete time setting, a stronger result in the spirit of Corollary \ref{corollary log2 decay} can be obtained under the additional finite path length assumption, i.e.\ if 
\[
\sum_{n=1}^\infty \|x_n - x_{n+1}\| = \eta\,\sum_{n=1}^\infty \|\nabla f(x_n)\|< \infty.
\]
The finite path-length assumption is true for gradient-descent trajectories in finite-dimensional spaces due to \cite{gupta2021path}.

\subsection{Stochastic gradient descent with multiplicative noise} 

We briefly consider a generalization of the gradient descent scheme discussed previously to the case where only stochastic gradient estimates are available. For background on conditional expectations, see e.g.\ Chapter 8 in \cite{klenke2013probability}.

We assume that all random variables are defined on the same probability space, which remains implicit in our analysis and is only required to be large enough to support sufficiently many independent variables. We assume that a (random) initial condition $X_0$ is given, and that we have access to a random variable $g_0$ such that $\E\big[ g_0 | \sigma(X_0)\big] = \nabla f(X_0)$ where $\sigma(X_0)$ denotes the $\sigma$-algebra generated by $X_0$. We can thus take a first step in the stochastic gradient procedure $X_1 = X_0 - \eta g_0$.

More generally, after taking $n$ steps, we set $\F_n = \sigma(X_0, \dots X_n)$ and assume that we again have a gradient estimate $g_n$ such that $\E[g_n|\F_n] = \nabla f(X_n)$, so that we can define $X_{n+1} = X_n - \eta g_n$.

An additional assumption is required in analyses of stochastic gradient descent to quantify the oscillations of $g_n$ around its mean. We make the multiplicative noise scaling assumption 
\[
\E\big[\|g_n - \nabla f(X_n)\|^2 \,\big|\,\F_n\big] \leq \sigma^2 \,\|\nabla f(X_n)\|^2.
\]
This assumption has recently gained popularity to model the overparametrized regime in deep learning with mean squared error \cite{DBLP:journals/corr/abs-1811-02564, wojtowytsch2021stochasticdiscrete, noisygradientsacceleration}.

\begin{lemma}
Assume that $f$ is convex, $\nabla f$ is $L$-Lipschitz continuous with respect to the Euclidean norm and consider the SGD trajectory with estimators $g_n$ as described above and a step size $\eta < \frac1{L(1+\sigma^2)}$. 

Under these conditions, $\E\big[f(X_n)\big]$ is monotone decreasing. If $f(x^*) = \inf f$, then 
\[
\eta \sum_{n=0}^\infty \E\big[ f(X_n) - \inf f\big] \leq \frac{\E\big[\|X_0-x^*\|^2\big]}2 + \frac{\eta(1+\sigma^2)}{1- L(1+\sigma^2)\eta/2} \,\E\big[f(X_0)-\inf f\big].
\]
\end{lemma}

In particular, with the choice $\eta = 1/(L(1+\sigma^2))$, which is optimal in terms of proving decay, this becomes
\[
\sum_{n=0}^\infty \E\big[ f(X_n) - \inf f\big] \leq \frac{ L(1+\sigma^2)}2 \,\E\big[ \|X_0-x^*\|^2\big] + 2(1+\sigma^2)\,\E\big[f(X_0) - \inf f\big]
\]
The set-up and step 0 in the proof are well-known in stochastic optimization --- many details can be found e.g.\ in \cite{wojtowytsch2021stochasticdiscrete, noisygradientsacceleration} and the references cited there. %We present the full argument for the readers' convenience.

We note that it would also be possible to consider a random variable $X^*$ which is $\F_0$-measurable such that $f(X^*) \equiv 0$ almost surely. In convex functions which do not admit a unique minimizer, this allows us to select the unique closest point projection of $X_0$ onto the closed and convex set of minimizers. The $\sigma$-algebras do not change under this modification, but the constant may be reduced significantly.

\begin{proof}
% {\bf Technical set-up.} By the tower identity of conditional expectations \cite[Theorem 8.14]{klenke2013probability}, we note that
% \begin{align*}
% \E[g_n\cdot (X_n-x^*)] 
%     &= \E\big[\E\big[g_n\cdot (X_n-x^*)\,\big|\,\F_n\big]\big]
%     = \E\big[\E\big[g_n\,\big|\,\F_n\big]\cdot (X_n-x^*)\big]
%     = \E\big[\nabla f(X_n)\cdot (X_n-x^*)\big]
% \end{align*}
% since $X_n-x^*$ is $\F_n$-measurable. The same is true with $\nabla f(X_n)$ in place of $X_n-x^*$. Furthermore
% \begin{align*}
% \E\big[\|g_n\|^2\big] 
%     &= \E\big[\|g_n-\nabla f(X_n)\|^2 + 2\,\langle g_n-\nabla f(X_n), \nabla f(X_n)\rangle + \|\nabla f(X_n)\|^2\big] \\
%     &= \E\big[ \E\big[ \|g_n-\nabla f(X_n)\|^2 \,\big|\F_n\big]\big] + 2\, \E\big[\E\big[\langle g_n-\nabla f(X_n), \nabla f(X_n)\rangle \,\big|\, \F_n\big]\big] +  \E\big[ \|\nabla f(X_n)\|^2\big]\\
%     &\leq \sigma^2\,\E\big[ \|\nabla f(X_n)\|^2 \big] + \E\big[\big\langle \E\big[g_n-\nabla f(X_n)\,\big|\, \F_n\big], \nabla f(X_n)\big\rangle \big] +  \E\big[ \|\nabla f(X_n)\|^2\big]\\
%     &= (1+\sigma^2) \E\big[ \|\nabla f(X_n)\|^2\big].
% \end{align*}

{\bf Step 1. Energy dissipation.} In the stochastic setting, the expected energy dissipation identity
\[
\E[f(x_{n+1}] \leq \E[f(x_n)]- \eta\left(1- \frac{L\eta}2\right) \E[\|\nabla f(x_n)\|^2]
\]
holds if the noise scales multiplicatively \cite[Lemma 16]{noisygradientsacceleration}.

% Consider any fixed direction $g$. Then 
% \begin{align*}
% f(x-\eta g) &= f(x) - \int_0^\eta \frac{d}{dt} f(x-t g)\dt\\
%     &= f(x) - \int_0^\eta \big\{\nabla f(x-tg) -\nabla f(x)\}\cdot g + \nabla f(x)\cdot g\dt\\
%     &\leq f(x) - \eta \nabla f(x)\cdot g + \int_0^\eta \|\nabla f(x-tg)-\nabla f(x)\|\cdot\|g\|\dt\\
%     &\leq f(x) - \eta\,\nabla f(x)\cdot g + L\|g\|^2\int_0^\eta t\dt\\
%     &\leq f(x) - \eta \,\nabla f(x)\cdot g + \frac{L\eta^2\|g\|^2}2
% \end{align*}
% Taking the expectation over $f(X_n-\eta g_n)$, we find that 
% \[
% \E\big[f(X_{n+1})\big] \leq \E\big[f(X_n)\big] - \eta\left(1- \frac{L(1+\sigma^2)\eta}2\right)\,\E\big[\|\nabla f(X_n)\|^2\big].
% \]
% In particular, if $\eta< 2/L(1+\sigma^2)$, then the sequence $\E\big[ f(X_n)\big]$ is decreasing. Furthermore, we deduce that
% \[
% \E\big[ f(X_0) - \inf f\big] \geq \sum_{n=0}^\infty \E\big[f(X_n) - f(X_{n+1})\big] \geq \eta \left(1- \frac{L(1+\sigma^2)\eta}2\right)\sum_{n=0}^\infty \E\big[\|\nabla f(X_n)\|^2\big].
% \]

{\bf Step 2. Summability.}
Consider the sequence $L_n:= \E\big[ \|X_n-x^*\|^2\big]$. We find that
\begin{align*}
    L_{n+1} &= \E\big[ \|X_{n+1}-X_n + X_n-x^*\|^2\big]\\
        &= \E\big[ \|X_{n+1} - X_n\|^2\big] + 2\,\E\big[\langle X_{n+1}-X_n, \,X_n-x^*\rangle \big] + L_n\\
        &= \eta^2\,\E\big[ \|g_n\|^2\big] - 2\eta\,\E\big[ \langle g_n, X_n-x^*\rangle\big]+ L_n\\
        &\leq \eta^2(1+\sigma^2)\,\E\big[\|\nabla f(X_n)\|^2\big] - 2\eta\,\E\big[ \langle \nabla f(X_n), X_n-x^*\rangle\big] + L_n\\
        &\leq \eta^2(1+\sigma^2)\,\E\big[\|\nabla f(X_n)\|^2\big] - 2\eta \E\big[ f(X_n) - f(x^*)\big] + L_n.
\end{align*}
Since $0\leq L_n$ for all $n$ we see that
\[
L_0 \geq L_0 - L_{n+1} = \sum_{k=0}^n (L_k - L_{k+1}) \geq2\eta \sum_{k=0}^n\E\big[ f(X_n) - f(x^*)\big] - \eta^2(1+\sigma^2) \sum_{k=0}^n \E\big[\|\nabla f(X_n)\|^2\big]
\]
and hence
\begin{align*}
\eta \sum_{k=0}^n \E\big[ f(X_n) - \inf f\big]
    &\leq \frac{L_0}2 + \frac{\eta^2(1+\sigma^2)}{\eta (1- \frac{L(1+\sigma^2)\eta}2)} \E\big[ f(X_0)-\inf f]\\
    &\leq \frac{\E\big[\|X_0-x^*\|^2\big]}2 + \frac{\eta(1+\sigma^2)}{1- L(1+\sigma^2)\eta/2} \,\E\big[f(X_0)-\inf f\big]
\end{align*}
\end{proof}

As a consequence of the summable decay, we immediately see that 
\[
\lim_{n\to\infty} n\cdot \E\big[ f(X_n) - \inf f\big] =0. 
\]
The more quantitative estimate
\[
\E[f(x_n) - f(x^*)] \leq (1+\sigma^2)\, \frac{\E[f(x_0) - f(x^*)] + \frac L2\,\E\big[\|x_0-x^*\|^2\big]}{ n+1+\sigma^2}.
\]
is derived in \cite[Lemma 7]{noisygradientsacceleration}.
As an application, we prove that the random variables $f(X_n)$ converge to $\inf f$ in a stronger fashion.

\begin{corollary}
Let $X_n$ be stochastic gradient iterates for $f$. If $\E[ f(X_0) + \|X_0\|^2] < \infty$, then $f(X_n)$ converges to $\inf f$ both almost surely and in $L^1$.
\end{corollary}

\begin{proof}
{\bf Convergence in $L^1$.} As $\sum_{n=0}^\infty \E\big[ f(X_n) - f(x^*)\big]$ converges, we find that $\lim_{n\to\infty} \E\big[ f(X_n) -\inf f\big] =0$. Since $f\geq \inf f$, this is the same as $L^1$-convergence
\[
\lim_{n\to\infty} \|f_n - \inf f\|_{L^1} = \lim_{n\to\infty} \E\big[ |f(X_n) - \inf f|\big] = \lim_{n\to\infty} \E\big[ f(X_n) -\inf f\big] =0.
\]

{\bf Convergence almost surely.} The fact that the sequence $\|f(X_n) - \inf f\|_{L^1} = \E\big[f(X_n) - f(x^*)\big]$ is summable implies convergence almost surely by \cite[Theorem 6.12]{klenke2013probability}.
\end{proof}

Without exploiting summability, we only get convergence almost surely for a subsequence $f(X_{n_k})$ of $ f(X_n)$ -- precisely such a sequence for which $\|f(X_{n_k}) - \inf f\|_{L^1}$ is summable.

\section{Momentum method}\label{section heavy ball}

\subsection{Continuous time heavy ball ODE}

In our analysis of accelerated gradient descent, let us consider the differential equation limit \cite{su2014differential}, given by
\begin{equation}\label{accelerated-diff-eq}
    \dot{x}(t) = v(t),~\dot{v}(t) = -\frac{\alpha}{t}v(t) - \nabla f(x_t).
\end{equation}
We have the following result concerning the integrability of the objective error $f(x_t) - f(x^*)$.

\begin{theorem}\label{theorem nesterov decay}
    Let $H$ be a Hilbert space and suppose that $f:H\rightarrow \mathbb{R}$ is a convex differentiable function and $x^*\in \arg\min_{x\in \mathbb{R}^d} f(x)$. Then if $x_t$ and $v(t)$ are defined via the differential equation \eqref{accelerated-diff-eq} with $\alpha > 3$, we have
    \begin{equation}
        \int_0^\infty t(f(x_t) - f(x^*))dt \leq \frac{(\alpha-1)^2\,\|x_0-x^*\|^2}{2(\alpha-3)}<\infty.
    \end{equation}
    The optimal bound is attained for $\alpha = 5$ where $\frac{(\alpha-1)^2}{2(\alpha-3)} = 4$.
\end{theorem}

This result was previoulsy obtained in \cite[Theorem 5]{su2014differential}. We provide a brief since we will use many of the same concepts below. In Section 4.2 of \cite{su2014differential}, the authors demonstrate that $\alpha< 2$ does not generally lead to energy decay even for quadratic functions and that $\alpha<3$ is generally inadmissible by considering $f(x) = |x|$.

In particular, `on average' $f(x_t)$ must decay faster than $O(1/t^2)$. A more quantitative statement does not follow immediately in this context since $f(x_t) - f(x^*)$ is not guaranteed to be monotone decreasing (and in many cases is not, see Section \ref{section nesterov oscillator}). However, it was observed in \cite{attouch2016rate} that $f(x_t) - f(x^*) = o(1/t^2)$ also for the discrete time Nesterov algorithm.

%Note that this strengthens the result from \cite{su2014differential}, which only states that $f(x_t) - f(x^*) = O(1/t^2)$.
\begin{proof}
    The proof closely follows the argument from \cite{su2014differential}. We consider the Lyapunov function
    \begin{equation}
        L(t) = t^2(f(x_t) - f(x^*)) + \frac{1}{2}\|(\alpha - 1)(x_t - x^*)  + tv(t)\|^2
    \end{equation}
    Differentiation this, we obtain
    \begin{equation}
    \begin{split}
        L'(t) = 2t(f(x_t) - f(x^*)) + t^2\langle \nabla f(x_t), v(t)\rangle - \langle t\nabla f(x_t), (\alpha - 1)(x_t - x^*)  + tv(t)\rangle,
        \end{split}
    \end{equation}
    since $\frac{d}{dt}[(\alpha - 1)(x_t - x^*)  + tv(t)] = -t\nabla f(x_t)$. Simplifying this and using the fact that convexity of $f$ means that
    $$
    \langle\nabla f(x_t), x_t - x^*\rangle \geq f(x_t) - f(x^*),
    $$
    we get
    \begin{equation}
        L'(t) \leq (3-\alpha)t(f(x_t) - f(x^*))
    \end{equation}
    Since $L(t) \geq 0$ for $t \geq 0$, we get that
    \begin{equation}
        \int_0^t t(f(x_t) - f(x^*))dt \leq \frac{L(0)}{\alpha - 3} < \infty.\qedhere
    \end{equation}
\end{proof}

The same proof immediately shows that
\[
t^2\big(f(x_t) - f(x^*)\big) \leq L(t) \leq L(0) \quad \Ra\quad f(x_t) - f(x^*) \leq \frac{(\alpha-1)^2\,\|x_0-x^*\|^2}{2t},
\]
where choosing the minimal value $\alpha =3$ yields the optimal bound.
We further note an improvement of `physical' nature. Namely, if $\alpha\geq 3$, then the total energy (potential and kinetic) of a particle $x_t$ with mass $m =1/2$ satisfies
\begin{align*}
t^2\bigg( f(x_t) - f(x^*) &+ \frac{m}2\|\dot x_t\|^2\bigg) = t^2\left( f(x_t) - f(x^*) \right) + \frac{1}4 \left\|t\dot x_t\right\|^2\\
    &\leq t^2\left( f(x_t) - f(x^*) \right) + \frac{1}2 \left\|t\dot x_t+(\alpha-1)(x_t-x^*)\right\|^2 + \frac{(\alpha-1)^2}2 \|x_t-x^*\|^2\\  &\leq L(t) + \frac{(\alpha-1)^2}2 \|x_t-x^*\|^2\\
    &\leq L(0) + \frac{(\alpha-1)^2}2 \|x_t-x^*\|^2.
\end{align*}
In particular, in any situation in which the sub-level sets of $f$ are compact, we find that not only the potential, but also the kinetic energy of $x_t$ decays at least as fast as $t^{-2}$.

\subsection{The Nesterov Oscillator}\label{section nesterov oscillator}

As a special example, we consider the `Nesterov oscillator' equation
\begin{equation}\label{eq nesterov oscillator}
     \ddot x = - \frac\alpha t\,\dot x - \mu x, \qquad x_0 = x_0, \qquad \dot x_0 = 0,
\end{equation}
based on a dampened harmonic oscillator with Nesterov type friction. This equation arises as the heavy ball ODE for the convex function $f(x) = \frac\mu2 \,x^2$. The coefficient of friction is not chosen optimally for this function, which is indeed strongly convex, but the example provides us with valuable intuition and tools for the proof below. 
 
The Nesterov oscillator has time-variable friction which transitions from the over-dampened to the under-dampened regime when
\[
 \left(\frac{\alpha}{2t}\right)^2 = \mu \qquad \LRa \quad t^2 = \frac{\alpha^2}{4\mu}.
\]
Recall that solutions to the classical harmonic oscillator equation
\[
     \ddot x = - \beta\,\dot x - \mu x, \qquad x_0 = x_0, \qquad \dot x_0 = 0
\]
are given by 
\[
x(t) = \exp\left(-\frac{\beta}2t \right)\left(c_1\,\cos(\omega t) + c_2\,\sin(\omega t)\right), \qquad \omega:= \sqrt{\mu - \left(\frac\beta2\right)^2}
\]
in the underdampened regime where $(\beta/2)^2<\mu$ and 
\[
x(t) = c_1\,\exp\left(\lambda_+ t\right) + c_2\,\exp(\lambda_- t), \qquad \lambda_\pm = -\frac\beta 2 \pm \sqrt{\left(\frac\beta2\right)^2-\mu}
\]
in the overdampened regime where $(\beta/2)^2>\mu$ (see e.g.\ \cite[Section 10.4]{konigsberger1999folgen} for a derivation). Overdampened solutions can overshoot the minimizer if the initial velocity is high enough, but at most once, and approach exponentially fast in a monotone fashion afterwards, while an underdamped oscillator changes sign an infinite number of times.

\begin{figure}
    \centering
    \includegraphics[width=0.245\textwidth]{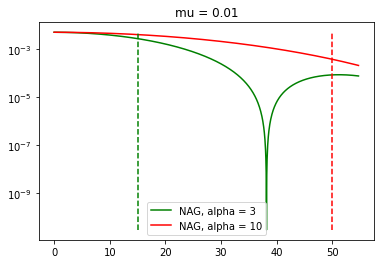}\hfill
    \includegraphics[width=0.245\textwidth]{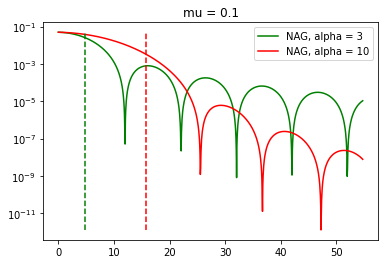}\hfill
    \includegraphics[width=0.245\textwidth]{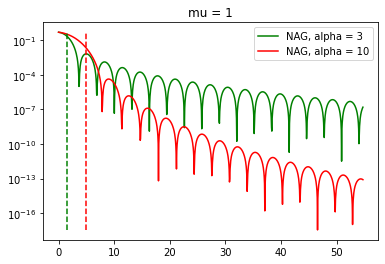}\hfill
    \includegraphics[width=0.245\textwidth]{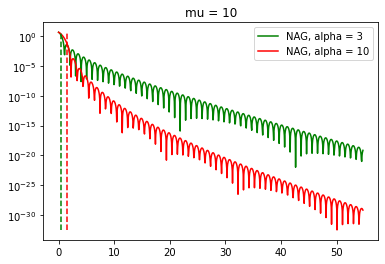}

    \caption{We consider numerical solutions of the Nesterov oscillator equation \eqref{eq nesterov oscillator} by the Nesterov algorithm \eqref{eq nesterov scheme} with step size $h=0.003$ and compare $f(x_t)$ for friction parameters $\alpha=3$ (green) and $\alpha =10$ (red) given the function $f(x) = \frac\mu2\,x^2$ where $\mu\in\{0.001, 0.1, 1, 10\}$ (left to right). The vertical dashed lines signal the critical time at which the oscillator transitions from the overdamped to the underdamped regime. We see that indeed, an initial slow overdamped period is followed by a period of oscillations whose length is essentially constant and only depends on $\mu$, but not $\alpha$. At the spikes, $f(x_t)$ behaves like $c^*\sqrt{|t-t^*|}$, since $\dot x_t\neq 0$, i.e.\ $x_t$ crosses the minimizer with positive velocity. If $\dot x_t$ were zero, the kinetic energy would be zero at the minimum, making the total energy zero and the system would be at equilibrium, meaning that there would be no crossing. This observation does not transfer to the higher-dimensional case.}
    \label{figure nesterov oscillator}
\end{figure}

With Nesterov friction, the coefficient of friction is initially infinitely strong. In the initial phase, we expect that solutions to the ODE gradually approach the minimizer $x=0$ in a monotone fashion. At time $t= \alpha/2\sqrt\mu$, the dampening changes type, and we expect the oscillator to change sign an infinite number of times with frequency approaching $\sqrt\mu$ as $t\to \infty$. In the long term, $f(x) = \frac\mu2\,x^2$ decays on average as $t^{-\alpha}$ according to \cite{aujol2019optimal}. More generally, the energy decays as $t^{-2\alpha/3}$ for general strongly convex functions due to \cite{attouch2018fast}. In the one-dimensional case, it crosses the global minimizer many times along this trajectory (at least in continuous time), but it does not come to rest. In the higher-dimensional case, the minimizer is not usually attained exactly at any finite time. More general results under more complicated curvature conditions are given in \cite{aujol2019optimal}.

A numerical investigation is given in Figure \ref{figure nesterov oscillator}, where we approximated $x(n\sqrt h)$ by the Nesterov scheme
\begin{equation}\label{eq nesterov scheme}
y_0 = x_0, \qquad x_{n+1} = y_n - h\,\nabla f(y_n), \qquad y_{n+1} = x_{n+1} + \frac{n}{n+\alpha} (x_{n+1}-x_n),
\end{equation}
which discretizes the heavy ball equation due to \cite{su2014differential}. 
Analytically, we focus on the initial overdamped phase. 
Since $x$ is not expected to change sign during this relatively brief period, we can make the ansatz
\[
x(t) = \exp\left(-\int_0^t \lambda(s)\ds\right)
\]
and compute
\[
\dot x = - \lambda\,x, \qquad \ddot x = (-\lambda' + \lambda^2) x
\]
such that
\[
0 \stackrel != \ddot x + \frac\alpha t\,\dot x + \mu x =  \left(-\lambda' + \lambda^2 - \frac{\alpha}t\lambda +\mu\right)x.
\]
The initial condition $0 = \dot x_0 = -\lambda(0)\,x_0$ induces the corresponding condition that $\lambda(0)=0$. We note that the functions
 \[
 \lambda_0(t) = 0, \qquad \lambda_+ (t) = \frac{\alpha}{2t}-  \sqrt{\left(\frac{\alpha}{2t}\right)^2 - \mu \:} 
 \]
 satisfy 
 \[
 \lambda_0^2 - \frac\alpha t\,\lambda_0 + \mu = \mu>0, \qquad \lambda_+^2 - \frac\alpha t\,\lambda_+ + \mu \equiv 0.
 \]
 In particular, by the comparison principle, we have $\lambda_0 \leq \lambda \leq \lambda_+$ for all $t\leq \frac{\alpha^2}{2\mu}$, at which point $\lambda_+$ is no longer defined. We conclude that
  \[
 x_+(t) := x_0\,\exp\left(\int_0^t\sqrt{\left(\frac{\alpha}{2s}\right)^2 - \mu\:} - \frac{\alpha}{2s}\ds\right)\qquad  \text{for } 0 < t < \frac{\alpha}{2\sqrt\mu}
 \]
 satisfies $|x_+(t)| \leq |x(t)|$.
Since $\sqrt{1-z}\geq 1-z$ for $z\in(0,1)$, we have
 \begin{align*}
 0 \leq \int_0^t \frac{\alpha}{2s} - \sqrt{\left(\frac\alpha{2s}\right)^2 - \mu}\ds
 	&= \int_0^t\frac{\alpha}{2s}\left(1 - \sqrt{1- \frac{4\mu s^2}{\alpha^2}}\right)\ds
	\leq \int_0^t \frac{\alpha}{2s} \cdot \frac{4\mu}{\alpha^2}s^2 \ds
	= \frac\mu{\alpha}\,t^2
 \end{align*}
for $t\in (0, \alpha/2\sqrt\mu)$. In particular, $\lim_{t\to 0}x_+(t) = x_0$. In particular, we see that
\[
|x(t)| \geq |x_+(t)| \geq |x_0|\,\exp\left(- \frac\mu\alpha\,t^2\right) \geq |x_0| \,\exp\left(-\frac\alpha4\right)\qquad\forall\ t\in\left(0, \frac{\alpha^2}{4\mu}\right).
\]
This settles the existence of an initial `slow' phase.

Notably, the heavy ball ODE with a momentum parameter scaling as $\alpha/t$ leads to much slower risk decay $t^{-2\alpha/3}$ in the strongly convex case than the gradient flow, which achieves the exponential rate $e^{-\mu t}$. With constant friction 
\[
\ddot x = - 2\sqrt\mu\,\dot x - \nabla f(x)
\]
where $\mu$ is the strong convexity constant of $f$, the decay is faster than the gradient flow with a decay as $e^{-\sqrt \mu\,t}$, see e.g.\ \cite[Theorem 1]{siegel2019accelerated}. This dampening is in fact optimal as seen in study of the harmonic oscillator.

% \todo[inline]{
% Not needed?
% }
% Assuming that $\lambda\in C^1[0,T^*)$ for the moment, we compute
% \[
% \lim_{t\to 0}\frac{\lambda(t)}t = \lim_{t\to 0}\lambda'(t) \quad\Ra \lambda'(0) = \lambda^2(0) - \alpha\,\lambda'(0) +\mu\quad \Ra \quad \lambda'(0) = \frac{\mu}{1+\alpha}.
% \]
% This enables us to use

\subsection{Heavy ball dynamics: Quadratic forms in Hilbert spaces}

We show that Theorem \ref{theorem nesterov decay} is essentially optimal by mimicking the argument of Lemma \ref{lemma slow convergence}.

\begin{lemma}\label{lemma nesterov lower bound}
Let $H$ be a separable Hilbert space. Then there exists a strictly convex quadratic function $F:H\to [0,\infty)$ such that $F(u) = 0$ if and only if $u=0$ and such the following is true: For any monotone decreasing function $g:(0,\infty)\to(0,\infty)$ such that 
\[
\int_0^\infty t\,g(t) \dt<\infty,
\]
there exists a solution $u$ to the heavy ball ODE 
\[
\ddot u = - \frac\alpha t\,\dot u - \nabla F(u)
\]
such that $F(u(t))\geq g(t) \quad\forall\ t\geq 1$.
\end{lemma} 

\begin{proof}
    {\bf Step 1.} Again, we identify $H$ isometrically with $L^2(1,\infty)$ and consider
    \[
    F:H\to\R, \qquad F(u) = \frac12\int_1^\infty \frac{u^2(s)}s\ds.
    \]
    The heavy-ball ODE acts pointwise in $s$, so by the analysis for the Nesterov oscillator with $\mu = 1/s$, we see that
    \[
    u(t)(s) \geq \exp(-\alpha/4) \,u_0(s) \qquad\forall\ s>0, \: t\in \left(0, \frac{\alpha \sqrt s}2\right).
    \]
    In particular, we have
    \[
    F(u(t)) \geq \exp(-\alpha/4) \int_{4t^2/\alpha^2}^\infty \frac{u_0^2(s)}s\ds 
    \]
    
    {\bf Step 2.} From the proof of Lemma \ref{lemma slow convergence}, we recall that for any monotone decreasing integrable function $\tilde g:\R\to\R$, there exists $u_0\in H$ such that
    \[
        \int_t^\infty \frac{u_0^s(s)}s\ds \geq \tilde g(t).
    \]
    Combining these arguments, we note that for any monotone decreasing integrable function $\tilde g:(0,\infty)\to (0,\infty)$ there exists $u_0\in H$ such that 
    \[
    F(u(t)) \geq \exp(-\alpha/4)\int_{4t^2/\alpha^2}^\infty\frac{u_0^2(s)}s\ds \geq \tilde g(t^2).
    \]
    Now note that 
    \[
    \int_0^\infty g(\sqrt t)\dt = 2\int_0^\infty g(\sqrt t)\,\frac{\sqrt t}{2\sqrt t}\dt = 2 \int_0^\infty g(s)\,s\ds,
    \]
    i.e.\ $\tilde g(t):= g(\sqrt t)$ is integrable if and only if $t\cdot g(t)$ is integrable. The decrease condition concerns $\tilde g$ and thus $g$, not $t\cdot g$.
\end{proof}

A slight mismatch remains between Theorem \ref{theorem nesterov decay} and Lemma \ref{lemma nesterov lower bound}: We have shown that $t\cdot (f(x_t) - \inf f)$ is integrable (but possibly non-monotone) and on the other hand that $f(x_t) -\inf f$ can decay as slowly as any {\em monotone} function for which $\int_0^\infty t(f(x_t) - \inf f)\dt < +\infty$.

\subsection{Heavy ball dynamics: No minimizers}

In Section \ref{section no minimizers}, we have shown that the risk decay along a gradient flow can be arbitrarily slow for convex functions without minimizers. Here, we prove the same for heavy ball dynamics.

%are generally no better than $x^{GD}(t^2/2)$, where $x^{GD}$ is a solution to the gradient flow equation. In particular, also solutions to the heavy ball ODE may exhibit an arbitrarily slow decrease of risk. 

%While our arguments are one-dimensional, we believe the observations to remain valid in more general settings.

\begin{lemma}
Let $g:[0,\infty)\to\R$ be a convex function such that $\lim_{x\to\infty} g(x) = 0$. Then for any $\alpha>0$, the solution to the Nesterov ODE
\[
\ddot x = - \frac\alpha t- f'(x), \quad x_0 = 0, \quad \dot x_0 = 0, \qquad f(x) = g\left(\frac{x}{2\,\sqrt{g(0)}}\right)
\]
satisfies $f(x(t)) \geq g(t)$. 
\end{lemma}

The strategy of proof is as follows: There is only so much total (kinetic and potential) energy in the system, and it is dissipated by the dynamics. Even if energy were conserved and totally translated into kinetic energy, this would bound the speed at which we move towards $\pm \infty$. Thus, if the function $x\mapsto f(x)$ decays slowly, so does $t\mapsto f(x_t)$.

\begin{proof}
Let $f:\R^d\to\R$ be any convex function and $x$ a solution to the Nesterov ODE 
\[
\ddot x = -\frac\alpha t\,\dot x - \nabla f(x), \qquad x_0 = x_0, \qquad\dot x_0 = 0.
\]
Note that
\[
\frac{d}{dt} \left(\frac{|\dot x|^2}2 + f(x) \right) = \dot x\,\ddot x + \nabla (x)\,\dot x = \left(\ddot x+ \nabla f(x)\right) \cdot \dot x = - \frac{\alpha}t \,\|\dot x\|^2\leq 0.
\]
In particular 
\[
\|\dot x\|^2 \leq 2\left(f(x) + \frac{\|\dot x\|^2}2\right) \leq 2\,f(x_0) \quad \Ra\quad \|x(t) - x_0\| \leq \int_0^t \|\dot x(s)\|\ds \leq \sqrt{2\,f(x_0)} t
\]
and hence
\[
f(x(t)) \geq \inf\left\{ f(z) : z\in B\big(x_0, \sqrt{2\,f(x_0)}t\big)\right\}.
\]
If $f:[0,\infty)\to \R$ is such that $\lim_{x\to \infty}f(x) = 0$, then $0= \inf f$ and $f$ is strictly monotone decreasing.
Hence, if $x_0=0$, then
\[
f(x(t)) \geq f\left(\sqrt{2\,f(0)}\, t\right).\qedhere
\]
\end{proof}

We have seen in the proof of Lemma \ref{lemma gf no minimizers} that there is no difference between the decay rate at infinity achievable by monotone decreasing or convex functions. In particular, also solutions to the heavy ball ODE with Nesterov momentum can decay arbitrarily slowly at infinity.

\appendix

\section{A comparison principle for convex functions with derivatives in \texorpdfstring{$L^{1/2}$}{L1/2}}\label{apppendix l12 derivative}

\lemmaa

The proof \sw{barely uses} properties of the square root function other than its concavity. We focus on this case for simplicity, but \sw{we believe that} the claim remains valid for general concave functions of the (negative) derivative. \sw{The claim is obvious for instance if we replace $\sqrt{-G'}$ by $-G'$ itself, since the integral only depends on $G(0)$ and $\lim_{z\to\infty}G(z)$. A property of the type $\sqrt{\cdot}' = \infty$ at $0$ is used for technical purposes in Step 3 of the variational proof.}

\sw{As the proof is somewhat technical, we present two proofs of Lemma \ref{square-root-integral-lower-bound} of very different flavor: One combinatorial, one variational, in order to make the result approachable to readers from different backgrounds.} The key to the \sw{combinatorial} proof will be the following well-known Proposition from the theory of majorization (see for instance \cite{marshall1979inequalities}).
\begin{proposition}\label{majorization-proposition}
    Let $a_1 \geq a_2 \geq \cdots a_n \geq 0$ and $b_1 \geq b_2 \geq \cdots b_n \geq 0$ be decreasing sequences which satisfy
    \begin{equation}\label{decreasing-sequence-condition}
        \sum_{j=i}^n b_j \geq \sum_{j=i}^n a_j
    \end{equation}
    for all $i=1,...,n$. Then the sequence $b_i$ is pointwise greater than an average of permutations of the sequence $a_i$. Specifically, letting $S_n$ denote the symmetric group on $n$ elements, there exists a map $\alpha:S_n\rightarrow \mathbb{R}_{\geq 0}$ satisfying
    \begin{equation}
        \sum_{\pi\in S_n} \alpha(\pi) = 1,
    \end{equation}
    such that
    \begin{equation}
        b_i \geq \sum_{\pi\in S_n}\alpha(\pi)a_{\pi(i)}.
    \end{equation}
    for $i=1,...,n$.
\end{proposition}
Although this Proposition is well-known, we give the simple proof for the reader's convenience.
\begin{proof}
    When $n=1$ the statement is trivial. We proceed by induction on $n$. 
    
    Suppose first that $b_i \geq a_i$ for $i=2,...,n$. Let $\tau_i\in S_n$ denote the transposition which swaps the $i$-th element and the first element for $i=2,...,n$ and let $e\in S_n$ denote the identity permutation. If $b_1 \geq a_1$, we can simply take $\alpha(e) = 1$ and $\alpha = 0$ for all other permutations. 
    
    Otherwise, we must have $a_1 > b_1 \geq a_i$ for all $i \geq 2$. In this case, the condition \eqref{decreasing-sequence-condition} implies that
    \begin{equation}
        \sum_{i=2}^n (b_i - a_i) \geq a_1 - b_1.
    \end{equation}
    Choose any numbers $0 \leq q_i \leq b_i - a_i$ such that $\sum_{i=2}^n q_i = a_1 - b_1$ and set
    \begin{equation}
        \alpha(\tau_i) = \frac{q_i}{a_1 - a_i}
    \end{equation}
    for $i=2,...,n$, $\alpha(e) = 1 - \sum_{i=2}^n\alpha(\tau_i)$, and $\alpha = 0$ for all other permutations. We claim that $\alpha$ defined in this way is non-negative, for which we must check that
    \begin{equation}
        \sum_{i=2}^n\alpha(\tau_i) = \sum_{i=2}^n \frac{q_i}{a_1 - a_i} \leq \sum_{i=2}^n \frac{q_i}{a_1 - b_1} = 1, 
    \end{equation}
    which follows from the choice of $q_i$. For each $i=2,...,n$ we calculate that
    \begin{equation}
        \sum_{\pi\in S_n}\alpha(\pi)a_{\pi(i)} = a_i + \alpha(\tau_i)(a_1 - a_i) = a_i + q_i \leq b_i,
    \end{equation}
    since $q_i \leq b_i - a_i$. For $i=1$ we calculate
    \begin{equation}
        \sum_{\pi\in S_n}\alpha(\pi)a_{\pi(1)} = a_1 - \sum_{i=2}^n\alpha(\tau_i)(a_1 - a_i) = a_1 - \sum_{i=2}^n q_i = b_1.
    \end{equation}
    This completes the proof in the case where $b_i \geq a_i$ for $i=2,...,n$.
    
    For the general case, the inductive assumption applied to the tail sequences $a_2,...,a_n$ and $b_2,...,b_n$ implies that there exists a map $\gamma:S_{n}\rightarrow \mathbb{R}_{\geq 0}$ such that $\sum_{\pi\in S_n} \gamma(\pi) = 1$, $\gamma(\pi) = 0$ if $\pi(1) \neq 1$ (i.e. the map is supported on the set of permutations of $\{2,...n\}$), and
    \begin{equation}
        b_i \geq \sum_{\pi\in S_n}\gamma(\pi)a_{\pi(i)}
    \end{equation}
    for $i=2,...,n$. Set
    $$
        \tilde{a}_i := \sum_{\pi\in S_n}\gamma(\pi)a_{\pi(i)}
    $$
    for $i=1,...,n$ (note that $\tilde{a}_1 = a_1$ since $\gamma(\pi) = 0$ if $\pi(1) \neq 1$).
    We complete the proof by apply the previous part to the sequences $b_1,...,b_n$ and $\tilde{a}_1,\tilde{a}_2,...,\tilde{a}_n$ (denoting the resulting averaging map by $\beta$) and noting that for each $i=1,...,n$
    \begin{equation}
    \begin{split}
        b_i \geq \sum_{\sigma\in S_n}\beta(\sigma)\tilde{a}_{\sigma(i)} &= \sum_{\sigma\in S_n}\beta(\sigma)\sum_{\pi\in S_n}\gamma(\pi)a_{\pi(\sigma(i))}\\
        &= \sum_{\pi\in S_n}\left(\sum_{\sigma\in S_n}\beta(\sigma)\gamma(\pi\sigma^{-1})\right)a_{\pi(i)}.
    \end{split}
    \end{equation}
    Setting
    \begin{equation}
        \alpha(\pi) = \sum_{\sigma\in S_n}\beta(\sigma)\gamma(\pi\sigma^{-1}),
    \end{equation}
    noting that $\alpha(\pi) \geq 0$ since $\beta$ and $\gamma$ are non-negative, and that
    \begin{equation}
        \sum_{\pi\in S_n}\alpha(\pi) = \sum_{\sigma\in S_n}\beta(\sigma)\sum_{\pi\in S_n}\gamma(\pi\sigma^{-1}) = \sum_{\sigma\in S_n}\beta(\sigma) = 1
    \end{equation}
    completes the proof.
\end{proof}

\begin{proof}[\sw{Combinatorial} Proof of Lemma \ref{square-root-integral-lower-bound}]
    Note first that since $G$ and $g$ are convex, they are differentiable almost everywhere \cite{rockafellar1997convex}. By considering the interval $[\epsilon,\infty)$ with $\epsilon\rightarrow 0$ it suffices to consider the case where $-g'(0),-G'(0) < \infty$. 
    
    The result follows from the stronger statement that if $g,G$ are non-negative convex functions with $\lim_{t\rightarrow \infty} g(t) = \lim_{t\rightarrow \infty} G(t) = 0$ and $G(t) \geq g(t)$, then
    \[
    \int_0^\infty \sqrt{-g'(t)}\dt \leq \int_0^\infty \sqrt{-G'(t)}\dt.
    \]
    Assume to the contrary that there exist $g,G$ satisfying these assumptions but for which
    \[
    \int_0^\infty \sqrt{-g'(t)}\dt > \int_0^\infty \sqrt{-G'(t)}\dt.
    \]
    Since $\lim_{t\rightarrow \infty} G(t) = 0$ and $g \geq 0$, this implies that there exists a $T < \infty$ such that
    \begin{equation}
        \int_0^T \sqrt{-g'(t)}\dt + \sqrt{g(T)} > \int_0^T \sqrt{-G'(t)}\dt + \sqrt{G(T)}.
    \end{equation}
    Since $G(T) \geq g(T)$ by assumption, we also have
    \begin{equation}\label{g-G-integral-bound-1352}
        \int_0^T \sqrt{-g'(t)}\dt + c\sqrt{g(T)} > \int_0^T \sqrt{-G'(t)}\dt + c\sqrt{G(T)}
    \end{equation}
    for any $c\leq 1$. 
    
    The functions $-g'$, $-G'$, $\sqrt{-g'}$, and $\sqrt{-G'}$ are bounded and decreasing (due to convexity), and so are Riemann integrable on $[0,T]$. This means that for any $\epsilon > 0$ we can choose an $N$ sufficiently large so that
    \begin{equation}\label{riemann-integral-bound}
        \frac{1}{N}\sum_{i=1}^{N} f(Ti/N) - \epsilon \leq \int_0^T f\dt \leq \frac{1}{N}\sum_{i=0}^{N-1} f(Ti/N) + \epsilon
    \end{equation}
    for each of the functions $f = -g',-G',\sqrt{-g'},\sqrt{-G'}$.
    
    Since the function $-g'(t)$ is decreasing, we have
    $$
        \frac{-1}{N}g'(T(i+1)/N) \leq \int_{Ti/N}^{T(i+1)/N} -g'(t)dt \leq \frac{-1}{N}g'(Ti/N).
    $$
    Combined with the estimate \eqref{riemann-integral-bound}, this implies that
    \begin{equation}\label{integral-sum-comparison-square-root}
        \left|\int_{0}^T\sqrt{-g'(t)}\dt - \sum_{i=0}^{N-1} \frac{1}{\sqrt{N}}\sqrt{\int_{Ti/N}^{T(i+1)/N} -g'(t)dt}\right| < \epsilon+\frac{1}{N}\sup_{t > 0}\sqrt{-g'(t)}.
    \end{equation}
    We obtain an analogous bound for $G$. 
    
    Consider two sequences defined by
    \begin{equation}
    \begin{split}
        a_i = \int_{T(i-1)/N}^{Ti/N} -g'(t)dt,~i=1,...,N,~a_{N+1} = \int_T^\infty -g'(t)dt = g(T)\\
        b_i = \int_{T(i-1)/N}^{Ti/N} -G'(t)dt,~i=1,...,N,~b_{N+1} = \int_T^\infty -G'(t)dt = G(T).
    \end{split}
    \end{equation}
    The sequences $a_i$ and $b_i$ are decreasing by the convexity of $g$ and $G$ and since $G(t) \geq g(t)$ we have
    \begin{equation}
        \sum_{j=i}^{N+1} b_j \geq \sum_{j=i}^{N+1} a_j
    \end{equation}
    for all $i=1,...,N+1$. However, using \eqref{g-G-integral-bound-1352} and \eqref{integral-sum-comparison-square-root}, choosing $c = 1/\sqrt{N}$ and $N$ large enough (i.e. $\epsilon$ small enough), we get
    \begin{equation}
        \sum_{i=1}^{N+1} \sqrt{b_i} < \sum_{i=1}^{N+1} \sqrt{a_i}. 
    \end{equation}
    However, this contradicts Proposition \ref{majorization-proposition} and Jensen's inequality, which completes the proof.
\end{proof}

\begin{proof}[Variational Proof of Lemma \ref{square-root-integral-lower-bound}]
{\bf Preliminaries.} Since $g, G$ are convex, they are (twice) differentiable almost everywhere. Since $\lim_{x\to\infty}g(x) = \lim_{x\to\infty}G(x) = 0$, they are monotone decreasing. Taking both together, we see that i.e.\ the integrands $\sqrt{-g'}, \sqrt{-G'}$ is well-defined. Since $\sqrt{-g'}, \sqrt{-G'}$ are monotone decreasing functions, they are (Riemann and Lebesgue) integrable on finite intervals. Their integrals over $(0,\infty)$ are well-defined (albeit potentially infinite).

{\bf Step 1. Reduction.} We implicitly assume that $g(t)>0$ for all $t>0$ since otherwise $g\equiv 0$ on an interval $[T, \infty)$ and thus $\int_0^\infty\sqrt{-g'(t)}\dt < \infty$. Up to a rescaling and translation, we may assume that $G\geq g>0$ on $[0,\infty)$ and that $g(0)$ and $g'(0)$ are finite. Under these stronger assumptions, we prove the stronger statement that 
\begin{equation}\label{eq G geq g}
\int_0^\infty \sqrt{-g'(t)}\dt \leq \int_0^\infty \sqrt{-G'(t)}\dt.
\end{equation}

{\bf Step 2. Representation by derivatives.} 
We note that
\[
g(t) = -\int_t^\infty g'(s)\ds, \qquad G(t) = -\int_t^\infty G'(s)\ds.
\]
and re-write the problem in terms of the non-negative monotone decreasing functions $\phi := -g'$ and $\psi:=-G'$. For this, we consider the convex set
\[
K_\phi:= \left\{\psi\in L^1(0,\infty) : \psi\text{ monotone decreasing}, \:\:\int_t^\infty \psi\ds \geq \int_t^\infty \phi\ds\quad\forall\ t>0\right\}
\]
and the `energy' functional
\[
E:K_\phi \to (0,\infty), \qquad E(\psi) = \int_0^\infty \sqrt\psi\dt.
\]
For the bijection between $G$ and $\psi = -G'$, denote $G_\psi(t) = \int_t^\infty \psi(s)\ds$.

{\bf Intermezzo: Proof strategy.} We will show that the functional $E$ has a minimizer $\psi^*$ in $K_\phi$ (Step 3). By constructing energy competitors, we will argue that if $G_\psi>g$ somewhere, then there exists $\tilde\psi$ such that $E(\tilde\psi)<E(\psi)$, i.e.\ $\psi \neq \psi^*$. The construction of energy competitors is the content of Step 4. This immediately implies that $\psi^*$ is such that $G_{\psi^*} = g = G_\phi$ everywhere on $(0,\infty)$, which concludes the proof. 

{\bf Step 3. Existence of minimizers.} If there exists no $\psi \in K_\phi$ such that $E(\psi)<+\infty$, then any function is a minimizer. In particular, the Lemma main statement of the Lemma, which can be phrased as: 
\[
E(\phi) = +\infty \qquad\Ra\quad E(\psi) = +\infty \quad\forall\ \psi \in K_\phi,
\]
holds. We now consider the case where there exists some $\psi \in K_\phi$ such that $E(\psi)<+\infty$. In this step, we prove that then there also exists $\psi^*\in K_\phi$ such that $E(\psi^*) = \inf_{\psi \in K_\phi} E(\psi)<+\infty$.

The existence of a minimizer is established by the direct method of the calculus of variations. Let $\psi_n\in K_\phi$ be a sequence such that $\lim_{n\to \infty} E(\psi_n) = \inf_{\psi \in K_\phi} E(\psi)$. In particular, we may assume that 
\[
E(\psi_n) \leq 1+ \inf_{\psi \in K_\phi} E(\psi)=: C \qquad \forall\ n\in\N.
\]
Then $|\{t: \psi_n(t)>R\}| \leq C/\sqrt{R}$ by Chebyshev's inequality and since $\psi$ is monotone decreasing, we find that $0\leq \psi_n(t) \leq R$ for $t>C/\sqrt{R}$ and for all $n\in\N$. Since $\psi_n$ is bounded on $[\eps,\infty)$ for all $\eps>0$ and monotone decreasing, we find that the sequence $\psi_n$ is bounded in $BV(a,b)$ for any $0<a<b<\infty$. In particular, there exists $\psi^*$ such that $\psi_n\to \psi^*$ in $L^p(a,b)$ for any $p<\infty$ due to the compact embedding of $BV$ into all Lebesgue spaces $L^p$ for finite $p$ in one dimension. We deduce that
\[
\int_a^b \sqrt{\psi^*(t)}\dt = \lim_{n\to\infty} \int_a^b\sqrt{\psi_n(t)}\dt \leq \liminf_{n\to\infty} \int_0^\infty \sqrt{\psi_n(t)}\dt.
\]
Since the inequality holds independently of $a,b$, we can send $a\to0$ and $b\to\infty$ to obtain 
\[
E(\psi^*) = \int_0^\infty \sqrt{\psi^*(t)}\dt\leq \liminf_{n\to\infty} \int_0^\infty \sqrt{\psi_n(t)}\dt = \inf_{\psi \in K_\phi} E(\phi) = \inf_{\psi \in K_\phi} E(\psi).
\]
It remains to show that $\psi^*\in K_\phi$. Observe first that for any $\psi\in K_\phi$ we have
\[
G_\psi(T) = \int_T^\infty \psi(t)\dt \leq \sqrt{\psi(T)} \int_T^\infty \sqrt{\psi(t)}\dt \leq \frac{2\int_{0}^\infty \sqrt{\psi(t)}\dt} {t}\cdot \int_0^\infty\sqrt{\psi(t)}\dt.
\]
In particular, for every $\eps>0$ there exists $T>0$ such that $G_{\psi}(T)<\eps$ for all $\psi \in K_\phi$ satisfying the energy bound $E(\psi)\leq C$. We conclude that 
\[
G_{\psi^*}(t) \geq \int_t^{T} \psi^*(t)\dt = \lim_{n\to\infty}\int_t^T \psi_n(t)\dt \geq \lim_{n\to\infty} \int_t^\infty \psi_n(t)\dt - \eps \geq \lim_{n\to\infty} g_{\psi_n}(t) - \eps \geq g(t) -\eps
\]
for all $t<T$. Since this holds for any $\eps>0$ and we can choose $T$ larger if we desire, we have $G_{\psi^*}(t)\geq g(t)$ for all $t\in(0,\infty)$.

%
%As above we note that if $E(\psi)\leq C$, then $0\leq \psi_n(t) \leq 1$ for $t\geq C$. In particular:
%\[
%C \geq E(\psi) \geq \int_0^\infty \sqrt\psi\ds \geq \int_C^\infty \sqrt\psi\ds \geq \int_C^\infty \psi \ds. 
%\]
%It follows that the sequence of point values $\gamma_n:= \int_{C}^\infty \psi_n\ds$ is bounded. We can therefore extract a convergent subsequence with limit $\gamma^*$. The fact that $\psi^*\in K_\phi$ follows from the observation that
%\[
%\int_t^\infty \psi^*\ds = \gamma^* -\int_C^t \psi^*\ds.
%\]
%Both terms on the right converge, meaning that $\int_t^\infty \psi^*\ds \geq \int_t^\infty \phi \ds$ for all $t$, i.e.\ $\psi \in K_\phi$.

{\bf Step 4. Identifying the minimizer.} In the following, we will show that $\psi^*=\phi$ to conclude the proof. This step will be partitioned into several arguments:
\begin{enumerate}
\item First, we show that $\psi^*$ must be piecewise constant in the set $\{G_{\psi^*}>g\}$.
\item Then, we show that $\psi^*$ is piecewise constant with at most one jump in connected components of $G_{\psi^*}>g$.
\item Finally, we see that the piecewise constant function is not energy-optimal unless $\psi^*=\phi$ (i.e., unless $\phi$ itself is piecewise constant with one jump).
\end{enumerate}
In every step, we require a similar but slightly different `energy competitor' argument.

{\bf Step 4.1. Step function structure.} The two functions
\[
g(t) = \int_t^\infty \phi(s)\ds , \qquad G^*(t) = \int_t^\infty \psi^*(s)\ds
\]
are continuous, so the coincidence set $I^* = \{t > 0 : g(t) = G^*(t)\}$ is closed. Assume that $t^*\in (0,\infty)\setminus I^*$, i.e. $G^*(t)>g(t)$ in an interval $(t^*-\eps, t^*+\eps)$. Since both functions are continuous and decreasing, we may assume that
\[
\inf_{t\in (t^*-\eps, t^*+\eps)} G^*(t) = G^*(t^*+\eps)> g(t^*-\eps) = \sup_{t\in(t^*-\eps, t^*+\eps)} g(t)
\]
for sufficiently small $\eps>0$.  This gives us great leeway to modify $\psi^*$ inside the interval $(t^*-\eps, t^*+\eps)$. We will show in this step that $\psi^*$ must be a step function with a single jump in $(t^*-\eps, t^*+\eps)$.

Namely, consider
\[
\tilde\psi(t) = \begin{cases}\psi^*(t^*-\eps) & t^*-\eps\leq t\leq t^\sharp\\ \psi^*(t^*+\eps) &t^\sharp \leq t \leq t^*+\eps\\
\psi^*(t) &\text{else}\end{cases}
\]
where $t^\sharp\in (t^*-\eps, t^*+\eps)$ is chosen such that
\[
\int_{t^*-\eps}^{t^*+\eps} \tilde\psi(t) \dt = \int_{t^*-\eps}^{t^*+\eps} \psi^*(t) \dt.
\]
In particular, we note that
\[
\tilde G(t) : = \int_t^\infty \tilde\psi(s)\ds 
\]
satisfies 
\[
\tilde G(t) \geq \tilde G(t^*+\eps) = G^*(t^*+\eps) > g(t^*-\eps) \geq g(t)\qquad \forall\ t\in (t^*-\eps, t^*+\eps)
\]
and $ G(t) = G^*(t)$ for all other $t$. In other words: $\tilde\psi\in K_\phi$. We claim that $E(\tilde\psi) \leq E(\psi^*)$ and that the inequality is strict unless $\tilde\psi \not\equiv \psi^*$. To see this, we write $\alpha = \psi^*(t^*)$, $\beta:= \psi^*(t^*-\eps)$ and
 \[
 \psi^*(t) = \lambda^*(t) \alpha + \big(1-\lambda^*(t)\big)\beta,\qquad
 \tilde \psi (t) = \tilde \lambda (t) \alpha + \big(1-\tilde\lambda(t)\big)\beta
 \]
 for functions $\tilde\lambda, \lambda^*:(t^*-\eps, t^*)\to\R$. The fact that $\int \tilde\psi = \int\psi^*$ is equivalent to observing that $\int\tilde\lambda = \int\lambda^*$. Since the square root function is concave, we have
 \begin{align*}
 \int_{t^*-\eps}^{t^*+\eps} \sqrt{\psi^*} \ds& = \int_{t^*-\eps}^{t^*+\eps} \sqrt{\lambda^*\alpha + (1-\lambda^*)\beta} \ds \geq \int_{t^*-\eps}^{t^*+\eps} \lambda^*\sqrt\alpha + (1-\lambda^*)\sqrt\beta \ds\\
 	& = \int_{t^*-\eps}^{t^*+\eps} \tilde\lambda\sqrt\alpha + (1-\tilde\lambda)\sqrt\beta \ds = \int_{t^*-\eps}^{t^*+\eps}\sqrt{\tilde\psi}\ds.
 \end{align*}
The inequality is strict unless $\alpha=\beta$ (which implies that $\psi\equiv \psi^*$) or $\lambda^*\in \{0,1\}$ almost everywhere. Since $\psi^*$ and thus also $\lambda^*$ is monotone, together with the integral constraint this means that $\psi^* \equiv \tilde\psi$. A strict inequality is excluded since $E(\psi^*)\leq E(\tilde\psi)$ since $\psi^*$ minimizes $E$ in $K_\phi$. Thus $\psi^*$ must be a s step function with only one step in $(t^*-\eps, t^*+\eps)$.
 
{\bf Step 4.2. Only one step.} Assume that $(a,b)$ is a connected component of the open set $\R\setminus I^*$. Step 4.1 shows that $\psi^*$ is a step function on $(a,b)$ with at most a finite number of jumps in any subinterval $(a',b')$ with $a< a' < b'<b$. If there were an {\em infinite} number of jumps, we could choose an accumulation point for  $t^*$ in Step 4.1 and obtain a contradiction. 

Assume that there are at least {\em two} jumps in $(a,b)$, i.e.\ there exist $a\leq t_1 < t_2 < t_3 < t_4\leq b$ and $z_1>z_2 >z_3$ such that
\[
\psi^*(t) = \begin{cases} z_1 &t_1<t<t_2\\ z_2 & t_2<t<t_3\\ z_3 &t_3 <t<t_4\end{cases}.
\]
We can see by the same strategy as in Step 4.1 that shifting $t_2$ right and $t_3$ left reduces the energy since $z_2$ is a convex combination of $z_1$ and $z_3$. The details are left to the reader. Since $G_{\psi^*}\geq g+\eps$ for some $\eps>0$ on the compact subset $[t_2, t_3]$, small perturbations of this type are admissible. By contradiction, we find that on every connected component $(a,b)$, $\psi^*$ can jump only once.

{\bf Step 4.3. No unbounded component.} Since $G\geq g >0$ and $\lim_{t\to\infty} G(t) = 0$, we note $G'(t)<0$ for all $t$. In particular, by the step function structure we find that $(0,\infty)\setminus I^*$ cannot have an unbounded connected component. 

{\bf Step 4.4. Coincidence at the origin.} We argue that without loss of generality, we may assume that $G(0) = g(0)$. If this is not the case, we can modify $g$ in the spirit of $G_r$ (see the passage after the statement of Lemma \ref{square-root-integral-lower-bound}) to increase $g(0)$ with an arbitrarily small increase in $\int\sqrt{-g'}$.

{\bf Step 4.5. Conclusion.} In this step, we will show that $(0,\infty)\setminus I^* = \emptyset$, i.e.\ that $G_{\psi^*}\equiv g$. Otherwise, this open set has at least one connected component. 

Assume that the interval $(a,b)$ is a connected component of $\R\setminus I^*$. Since $0\in I^*$ and there are no unbounded connected components, we find that $a, b\in I^*$. Thus 
\[
\int_t^b \psi^*(s) \ds > \int_t^b \phi(s)\ds \quad\forall\ t\in (a,b)\qquad\text{and}\quad \int_a^b \psi^*(s) \ds = \int_a^b \phi(s)\ds 
\]
and hence also
\[
 \int_a^t \psi^*(s)\ds < \int_a^t\phi(s)\ds \quad\forall\ t\in(a,b).
\]
For the minimizer $\psi^*(t) = \alpha\,1_{(a,t^\sharp)}(t) + \beta\,1_{(t^\sharp, b)}$ for $\alpha\geq \beta$. We immediately conclude that 
\[
\phi(a^+) := \lim_{t\searrow a} \phi(t) \geq \alpha, \qquad \phi(b^-):= \lim_{t\nearrow b}\phi(t) \leq \beta.
\]
We distinguish three cases. 

\begin{enumerate}
\item $\alpha = \phi(a^+)$. Since $\phi = -g'$ is monotone decreasing, we have
\[
G(t) = G(a) - \int_a^t \psi^*(s)\ds = g(a) - \int_a^t \alpha \ds  \leq g(a) - \int_a^t \phi(s)\ds = g(t) \qquad\forall\ t\in(a,t^\sharp).
\]
Since $G\geq g$, this implies that $G\equiv g$ in $(a, t^\sharp)$ and hence $[a,t^\sharp] \subseteq I$. This contradicts the choice of $(a,b)\subseteq (0,\infty)\setminus I^*$.

\item $\beta = \phi(b^-)$. A contradiction follows by the same argument.

\item So far, we have only used the monotonicity of $\phi$, the step function properties of $\psi^*$ and the fact that $g\leq G_{\psi^*}$. If we have strict inequalities
\[
\alpha < \phi(a^+), \qquad  \phi(b^-)< \beta,
\]
we are using an energy argument, i.e.\ we show that the $\psi^*$ cannot be energy optimal on $(a,b)$. This is not surprising -- we can increase $\alpha$ and decrease $\beta$ slightly without violating the constraints of our problem. As in previous arguments, this reduces the energy. The remainder of this proof is dedicated to the finer details of this argument.
\end{enumerate}

Let $\xi>0$ be such that $a+\xi < t^\sharp < b-\xi$ and consider an energy competitor 
\[
\psi^\eps(t) = \begin{cases} \alpha + \eps& a < t < a+\xi\\ \beta - \eps & b-\xi < t < b\\ \psi^*(t) &\text{else}\end{cases}.
\]
By construction, we have $G_{\psi^\eps} = G_{\psi^*}$ outside $(a,b)$ since $\int_a^b \psi^\eps = \int_a^b \psi^*$ by construction.
If $\eps>0$ is sufficiently small, then $\psi^\eps(a^+) < \phi(a^+)$ and $\psi^\eps(b^-) > \phi(b^-)$. By definition of $\phi(b^-)$, we conclude that $\psi^\eps(t)>\phi(t)$ for $t$ close to $b$ and thus 
\[
G_{\psi^\eps}(t) = \int_t^b \psi^\eps(s) \ds + G_{\psi^\eps}(b) = \int_t^b \psi^\eps(s) \ds + g(b) \geq \int_t^b\phi(s)\ds + g(b) = g(t).
\]
By a similar argument, we observe that $G_{\psi^\eps}(t) \geq g(t)$ in a neighbourhood of $a$. Without loss of generality, we choose $\eps_0, \delta>0$ such that $G_{\psi^\eps}\geq g$ on $(0, a+ \delta] \cup [b-\delta, \infty)$. On the compact interval $[a+\delta, b-\delta]$, we have 
\[
\min_{t\in[a+\delta, b-\delta]} (G_{\psi^*} - g)(t)>0 \qquad\Ra\quad G_{\psi^*}\geq g + \rho
\]
for some small $\rho>0$. By continuity, we conclude that $G_{\psi^\eps}\geq g$ also on $[a+\delta, b-\delta]$ for sufficiently small $\eps>0$. In total, we have shown that $\psi^\eps \in K_\phi$ for all sufficiently small $\eps$, i.e.\ that $\psi^\eps$ is a valid energy competitor. Since
\begin{align*}
\frac{d}{d\eps}\bigg|_{\eps=0} E(\psi^\eps) &=\frac{d}{d\eps}\bigg|_{\eps=0} \left(\int_a^{a+\xi} \sqrt{\alpha+\eps}\dt + \int_{b-\xi}^b \sqrt{\beta-\eps}\dt\right) = \xi\,\frac{d}{d\eps}\left( \sqrt{\alpha+\eps} + \sqrt{\beta-\eps}\right)\\
	&= \frac{\xi}2\left(\frac1{\sqrt\alpha}-\frac1{\sqrt\beta}\right)<0
\end{align*}
unless $\alpha = \beta$ since $\alpha \geq \beta$ due to the monotonicity of $\psi^*$. If $\alpha=\beta$, the first derivative vanishes: Constant $\psi^*$ is a {\em maximum} in the space of monotone decreasing functions:
\[
\frac{d^2}{d\eps^2}\bigg|_{\eps=0} E(\psi^\eps) = \frac{d}{d\eps}\bigg|_{\eps=0} \frac{\xi}2\left(\frac1{\sqrt{\alpha+\eps}}-\frac1{\sqrt{\alpha-\eps}}\right)= -\frac{\xi}2 \alpha^{-3/2}<0.\qedhere
\]

\end{proof}

\bibliographystyle{./alphaabbr}
\bibliography{./bibliography}

\end{document}